\documentclass[11pt]{amsart}
\usepackage{verbatim}
\usepackage{amsmath}
\usepackage{enumerate}
\usepackage{amssymb}
\usepackage{color}
\usepackage{url}
\usepackage{graphicx}
\usepackage{fullpage}
\usepackage{tikz-cd}
\usepackage{amsmath,amscd}

\newcommand\+{\;\lower\plusheight\hbox{$+$}\;}
\newcommand\lldots{\;\lower\plusheight\hbox{$\cdots$}\;}
\newcommand{\SL}{\mbox{SL}}
\newcommand{\GL}{\mbox{GL}}
\newcommand{\Gr}{\mbox{Gr}}
\newcommand{\Z}{\mathbb{Z}}
\newcommand{\Q}{\mathbb{Q}}
\newcommand{\HH}{\mathbb{H}}
\newcommand{\A}{\mathbb{A}}
\newcommand{\C}{\mathbb{C}}
\newcommand{\RR}{\mathbb{R}}

\newcommand{\D}{\mathbb{D}}

\newtheorem{Theorem}{Theorem}[section]
\newtheorem{Lemma}[Theorem]{Lemma}
\newtheorem{Corollary}[Theorem]{Corollary}
\newtheorem{Definition}[Theorem]{Definition}
\newtheorem{Example}[Theorem]{Example}
\newtheorem{Remark}[Theorem]{Remark}
\newtheorem{Proposition}[Theorem]{Proposition}

\def\Re{\mathop{\rm Re}\nolimits}

\def\Im{\mathop{\rm Im}\nolimits}

\setbox0=\hbox{$+$}
\newdimen\plusheight
\plusheight=\ht 0 \setbox0=\hbox{$-$}
\newdimen\minusheight
\minusheight=\ht0

\setbox0=\hbox{$\cdots$}
\newdimen\cdotsheight
\cdotsheight=\plusheight

\begin{document}

\title{Difference of a Hauptmodul for $\Gamma_{0}(N)$ and certain Gross--Zagier type CM value formulas}
\author{Dongxi Ye}
\address{
School of Mathematics (Zhuhai), Sun Yat-sen University, Zhuhai 519082, Guangdong,
People's Republic of China}
\email{yedx3@mail.sysu.edu.cn }
\subjclass[2010]{11F03; 11F27; 11F41; 11G15; 11G18}
\keywords{Borcherds product; the Gross--Zagier CM value formula; the Monster denominator formula;  Modular forms for the Weil representation}
\thanks{The author is supported by the Natural Science Foundation of China (grant No.11901586), the Natural Science Foundation of Guangdong Province (grant No.2019A1515011323) and the Sun Yat-sen University Research Grant for Youth Scholars (grant No.19lgpy244).}
\maketitle
\begin{abstract}
In this work, we show that the difference of a Hauptmodul for a genus zero group $\Gamma_{0}(N)$ as a modular function on $Y_{0}(N)\times Y_{0}(N)$ is a Borcherds lift of type $(2,2)$. As applications, we derive Monster denominator formula like product expansions for these modular functions and certain Gross--Zagier type CM value formulas.
\end{abstract}
\noindent
\numberwithin{equation}{section}
\allowdisplaybreaks
\section{Introduction}
\label{intro}

In his seminal work \cite{bor3,bor}, Borcherds develops a remarkable method to construct meromorphic modular forms $\Psi(z,f)$ on an orthogonal Shimura variety associated to some rational quadratic space of signature $(n,2)$ from some weakly holomorphic modular form $f$ for the Weil representation of ${\rm SL}_{2}(\Z)$ via regularizing an integral called theta-lift against the Siegel theta function. We now call such meromorphic modular forms Borcherds lifts of type $(n,2)$. Moreover, Borcherds shows that $\Psi(z,f)$ has a beautiful product representation called Borcherds product for its Fourier expansion near a cusp of the orthogonal Shimura variety. One of the most famous Borcherds lifts is the difference $j(\tau_{1})-j(\tau_{2})$ of the well known Klein's modular $j$-invariant as a modular function for $\SL_{2}(\Z)\times\SL_{2}(\Z)$, which is a Borcherds lift of type $(2,2)$ with $j-744$ as its theta-lift input, namely, $j(z_{1})-j(z_{2})=\Psi(z,j-744)$. If we write $j(\tau)=q^{-1}+\sum_{n\geq0}c(n)q^{n}$ with $q:=\exp(2\pi i\tau)$, then computing the Borcherds product of $\Psi(z,j-744)$ near the cusp of the underlying Shimura variety of type $(2,2)$, which is identified with the cusp $(i\infty,i\infty)$ of $Y(1)\times Y(1)$, we can recover the famous Monster denominator formula \cite{bor4}, namely,
\begin{equation}
\label{monst}
j(z_{1})-j(z_{2})=(q_{1}^{-1}-q_{2}^{-1})\prod_{m,n>0}\left(1-q_{1}^{m}q_{2}^{n}\right)^{c(mn)}.
\end{equation}
Furthermore, another interesting application of the fact that $j(z_{1})-j(z_{2})$ is a Borcherds lift $\Psi(z,j-~744)$ is to serve as a key ingredient  in giving a new proof to the interesting and famous Gross--Zagier CM value formula \cite{gz}, namely,
\begin{equation}
\label{GZj}
\left|{\rm Norm}\left(j\left(\frac{d_{1}+\sqrt{d_{1}}}{2}\right)-j\left(\frac{d_{2}+\sqrt{d_{2}}}{2}\right)\right)\right|^{\frac{8}{w_{1}w_{2}}}=\prod_{\substack{x,n,n'\in\mathbb{Z}\\n,n'>0\\x^{2}+4nn'=d_{1}d_{2}}}n^{\epsilon(n')},
\end{equation}
where $d_{1}$ and $d_{2}$ are coprime negative fundamental discriminants, $\epsilon(n')$ is multiplicative, and $\epsilon(p)$ is defined via the local Hilbert symbol at a prime $p$, using  the so called big CM value formula established by Bruinier, Kudla and Yang \cite{bky}. Such a formula was first discovered and proved by Gross and Zagier \cite{gz}, and it presents a beautiful and remarkable prime factorization formula for the rational norm of the algebraic integer $j\left(\frac{d_{1}+\sqrt{d_{1}}}{2}\right)-j\left(\frac{d_{2}+\sqrt{d_{2}}}{2}\right)$ and reveals the arithmetic information encoded in the exponents of the prime factors via Hilbert symbols. The new proof of the Gross--Zagier CM value formula as mentioned above has been recently worked out by Yang and Yin \cite{yangyin}, and it yields an equivalent form of \eqref{GZj} as follows.  
\begin{Theorem}[Gross and Zagier]
Let $d_{1},d_{2}$ be two coprime negative fundamental discriminants, and write $E_{i}=\mathbb{Q}(\sqrt{d_{i}})$ for $i\in\{1,2\}$. Let $F=\mathbb{Q}(\sqrt{D})$ with $D=d_{1}d_{2}$, and write $E=E_{1}E_{2}=\mathbb{Q}(\sqrt{d_{1}},\sqrt{d_{2}})$. Then
\begin{align}
&\sum_{([\mathfrak{a}_{1}],[\mathfrak{a}_{2}])\in {\rm Cl}(E_{1})\times{\rm Cl}(E_{2})}\log|j(\tau_{\mathfrak{a}_{1}})-j(\tau_{\mathfrak{a}_{2}})|^{\frac{8}{w_{1}w_{2}}}\nonumber\\
\label{equivgz}
&=\sum_{\substack{t=\frac{2m+D+\sqrt{D}}{2}\\ |2m+D|<\sqrt{D}\\m\in\Z}}\sum_{\text{$\mathfrak{p}$ inert in $E/F$}}\frac{1+{\rm ord}_{\mathfrak{p}}(t\mathcal{O}_{F})}{2}\rho(t\mathfrak{p}^{-1})\log(N(\mathfrak{p})),
\end{align}
where ${\rm Cl}(E_{i})$ is the ideal class group of $E_{i}$, $\tau_{\mathfrak{a}}=\frac{b+\sqrt{d}}{2a}$ is the unique CM point in the upper half plane $\mathbb{H}$ given by the integral ideal $\mathfrak{a}=[a,\frac{b+\sqrt{d}}{2}]$, $w_{i}$ is the number of roots of unity in $E_{i}$, and for an integral ideal $\mathfrak{a}$ of $F$
$$
\rho(\mathfrak{a}):=|\{\mathfrak{B}\subset\mathcal{O}_{E}:\,{\rm N}_{E/F}(\mathfrak{B})=\mathfrak{a}\}|=\prod_{\mathfrak{q}<\infty}\rho_{\mathfrak{q}}(\mathfrak{a})
$$
which can be computed via calculating its local factors $\rho_{\mathfrak{q}}(\mathfrak{a})$ defined by
\begin{equation}
\label{localrho}
\rho_{\mathfrak{q}}(\mathfrak{a})=\begin{cases}1&\text{if $\mathfrak{q}$ is ramified in $E$,}\\\frac{1+(-1)^{{\rm ord}_{\mathfrak{q}}\mathfrak{a}}}{2}&\text{if $\mathfrak{q}$ is inert in $E$,}\\1+{\rm ord}_{\mathfrak{q}}\mathfrak{a}&\text{if $\mathfrak{q}$ is split in $E$.}
\end{cases}
\end{equation}
\end{Theorem}
\begin{Remark}
To see how \eqref{equivgz} is equivalent to \eqref{GZj}, we refer the reader to \cite[Eq. (7.1)]{gz} and \cite[Remark~4.1]{yangyin}. Another way to obtain the Gross--Zagier CM value formula \eqref{equivgz} is to use the so-called small CM value formula due to Schofer \cite{S}, which is a generalization of \eqref{equivgz} to the average value of a Borcherds lift over a CM 0-cycle associated to an imaginary quadratic field.
\end{Remark}
As one will see in Section~2, the divisor of a Borcherds lift is a linear combination of some special divisors (see Subsection~\ref{specialdiv} for definitions) on a Shimura variety, which turns out to be a necessary condition for a meromorphic function on the Shimura variety to be a Borcherds lift. In addition, one will also see that $Y_{0}(N)\times Y_{0}(N)$ can be viewed as a Shimura variety $\mathcal{X}_{K_{N}}$ of dimension~2, and its diagonal can be interpreted as a special divisor. So the injectivity and surjectivity of a Hauptmodul for the genus zero Hecke subgroup $\Gamma_{0}(N)$ imply that its difference as a meromorphic function on $\mathcal{X}_{K_{N}}$ has a special divisor, which is potentially to be a Borcherds lift. All of these observations together with those interesting phenomena and relations mentioned above that are related to $j(\tau)$ have greatly motivated us to look into the genus zero Hecke subgroups $\Gamma_{0}(N)$ cases. Now recall that for a genus zero congruence subgroup $\Gamma$ of $\SL_{2}(\RR)$ commensurable with $\SL_{2}(\Z)$, the function field on $X(\Gamma)$ can be generated by a single modular function, and such function is called a Hauptmodul for $\Gamma$ if it has a unique simple pole of residue~1 at the cusp $i\infty$, i.e., it has a Fourier expansion of the form $q^{-1/h}+c(0)+c(1)q^{1/h}+\cdots$ with $q=\exp(2\pi i\tau)$ at the cusp $i\infty$ where $h$ is the width of the cusp $i\infty$. In this work, we first aim to extend the fact that $j(z_{1})-j(z_{2})$ is a Borcherds lift of type~$(2,2)$ to any genus zero groups $\Gamma_{0}(N)$ and show that the difference of a Hauptmodul $\pi_{N}(\tau)$ for a genus zero group $\Gamma_{0}(N)$, which is known \cite[Theorem 15, p.103]{Sc} to be for $N\leq10$ and $N\in\{12,13,16,18,25\}$, as a modular function for $\Gamma_{0}(N)\times \Gamma_{0}(N)$ is a Borcherds lift of type $(2,2)$ by explicitly constructing the lift input by a uniform approach.  These extend Scheithauer's results \cite{sch3}, in which he works only on $N$ square free using the twisted denominator identity of the Monster Lie algebra. The method we use is different from Scheithauer's and is more natural from the analytic point of view. Later, employing the big CM value formula \cite{bky} together with the fact that $\pi_{N}(z_{1})-\pi_{N}(z_{2})$ is a Borcherds lift shown in this work, we derive Gross--Zagier type CM value formulas for $\pi_{p}(\tau)$ for $p\in\{3,5,7,13\}$, the only odd primes such that $\Gamma_{0}(p)$ are of genus zero.


\begin{Remark}
The Hauptmoduls for $\Gamma_{0}(N)$ for $N\geq2$ are well known and have been constructed explicitly by eta-quotients. We record them here as a reference for the reader:
$$
 \begin{tabular}{|c | c| c |c | c| c | c | c | c | c| c | c| c|c|c|} 
 $N$ & $2$ & $3$ & $4$&$ 5$&$ 6 $&$ 7$ &$ 8$ & $9$ & $10$ & $12$ & $13$ &$16$&$18$&$25$\\
 $\pi_{N}(\tau)$ & $\frac{\eta_{1}^{24}}{\eta_{2}^{24}}$ & $\frac{\eta_{1}^{12}}{\eta_{3}^{12}}$ & $\frac{\eta_{1}^{8}}{\eta_{4}^{8}}$ & $\frac{\eta_{1}^{6}}{\eta_{5}^{6}}$ & $\frac{\eta_{3}\eta_{1}^{5}}{\eta_{2}\eta_{6}^{5}}$ & $\frac{\eta_{1}^{4}}{\eta_{7}^{4}}$ & $\frac{\eta_{1}^{4}\eta_{4}^{2}}{\eta_{2}^{2}\eta_{8}^{4}}$ & $\frac{\eta_{1}^{3}}{\eta_{9}^{3}}$ & $\frac{\eta_2 \eta_5^5}{\eta_1 \eta_{10}^5}$ &  $\frac{\eta _3^3 \eta _4}{\eta _1 \eta _{12}^3}$ & $\frac{\eta_{1}^{2}}{\eta_{13}^{2}}$ & $\frac{\eta _8 \eta _1^2}{\eta _2 \eta _{16}^2}$& $\frac{\eta_6 \eta_9^3}{\eta_3
                                        \eta_{18}^3}$&$\frac{\eta _1}{\eta _{25}}$
 \end{tabular}
 ,
 $$
 where $\eta_{m}:=\eta(m\tau)$, and $\eta(\tau)=q^{\frac{1}{24}}\prod_{n=1}^{\infty}(1-q^{n})$ denotes the Dedekind eta function.
 \end{Remark}


Now we state the first main result of this work.
\begin{Theorem}
\label{borthm}
Let $\pi_{N}(\tau)$ be a Hauptmodul for a genus zero group $\Gamma_{0}(N)$ for $N\geq2$. Then $\pi_{N}(z_{1})-\pi_{N}(z_{2})$ is a Borcherds lift $\Psi(z,F_{N})$ of type $(2,2)$ for some weakly holomorphic modular function $F_{N}$ for the Weil representation of ${\rm SL}_{2}(\Z)$.
\end{Theorem}

Define, here and throughout the remainder of this work, $\mathcal{C}(\Gamma_{1}(N))$ to be the set of inequivalent cusps $s={a_{s}}/{c_{s}}\in\mathbb{Q}$ of $\Gamma_{1}(N)$ with $(a_{s},c_{s})=1$, $M_{s}=\begin{pmatrix}a_{s}&b_{s}\\c_{s}&d_{s}\end{pmatrix}\in\SL_{2}(\mathbb{Z})$, $m_{s}=(c_{s},N)$, and write $h_{s}=N/m_{s}$.  Note that when $N\ne4$, $\Gamma_{1}(N)$ has no irregular cusps, and then $h_{s}=N/m_{s}$ is the width of cusp $s\in\mathcal{C}(\Gamma_{1}(N))$. If $s$ is regular write for $0\leq t<h_{s}$
$$
\left(\left.f\right|M_{s}\right)_{t}=\sum_{n\gg-\infty}A_{s}(nh_{s}+t)q^{(nh_{s}+t)/h_{s}}
$$
where
$$
\left.f\right|M_{s}=\sum_{n\gg-\infty}A_{s}(n)q^{n/h_{s}}\quad\mbox{and}\quad \left.f\right|M_{s}:=f\left(\frac{a_{s}\tau+b_{s}}{c_{s}\tau+d_{s}}\right).
$$
 Then we can derive a product representation for the Fourier expansion of $\pi_{N}(z_{1})-\pi_{N}(z_{2})$ at the cusp $(i\infty,i\infty)$ as follows, which can be viewed as analogues of the Monster denominator formula~\eqref{monst}.
\begin{Corollary}
\label{diffpiexp}
Let $\pi_{N}(\tau)$ be a Hauptmodul for a genus zero group $\Gamma_{0}(N)$
  and define for $d|N$,
\begin{align*}
&\sum_{\ell=-1}^{\infty}A(\ell,d)q^{\ell}\\
&=\frac{2}{\lambda_{2,N}|\Gamma_{0}(N):\Gamma_{1}(N)|}\left\{\sum_{\substack{s\in\mathcal{C}(\Gamma_{1}(N))\\\text{$s$ regular}\\m_{s}=d}}\left[\left(\left.\pi_{N}\right|M_{s}\right)_{0}-A_{s}(0)\right]+\sum_{\substack{s\in\mathcal{C}(\Gamma_{1}(N))\\\text{$s$ irregular}\\m_{s}=d}}\frac{1}{h_{s}}\left[\left.\pi_{N}\right|M_{s}-A_{s}(0)\right]\right\},
\end{align*}
where $\lambda_{2,N}=2$ or $1$ depending on whether $N=2$ or not.
Then we have
\begin{equation}
\label{prodana}
\pi_{N}(z_{1})-\pi_{N}(z_{2})=(q_{1}^{-1}-q_{2}^{-1})\prod_{m,n>0}\prod_{d|N}\left(1-\left(q_{1}^{m}q_{2}^{n}\right)^{\frac{N}{d}}\right)^{A(mn,d)}
\end{equation}
where $q_{j}=\exp(2\pi iz_{j})$.
\end{Corollary}
Here are two concrete examples following from Corollary \ref{diffpiexp}.

\begin{Example}
\label{rrbor}
Let $\pi_{5}(\tau)$ be the Hauptmodul for $\Gamma_{0}(5)$ given by
$$
\pi_{5}(\tau)=\left(\frac{\eta(\tau)}{\eta(5\tau)}\right)^{6},
$$
where $\eta(\tau)$ is the Dedekind eta function. Then
$$
\left(\frac{\eta(z_{1})}{\eta(5z_{1})}\right)^{6}-\left(\frac{\eta(z_{2})}{\eta(5z_{2})}\right)^{6}=(q_{1}^{-1}-q_{2}^{-1})\prod_{m,n>0}(1-q_{1}^{m}q_{2}^{n})^{A(mn,5)}(1-q_{1}^{5m}q_{2}^{5n})^{A(mn,1)}.
$$
where
$$
\sum_{\ell=-1}^{\infty}A(\ell,5)q^{\ell}=\left(\frac{\eta(\tau)}{\eta(5\tau)}\right)^{6}+6,
$$
$$
A(\ell,1)=a(5\ell)\quad\mbox{and}\quad\sum_{\ell=0}^{\infty}a(\ell)q^{\ell}=125\left(\frac{\eta(5\tau)}{\eta(\tau)}\right)^{6}.
$$
\end{Example}

\begin{Example}
\label{gorgo}
Let $\pi_{8}(\tau)$ be the Hauptmodul for $\Gamma_{0}(8)$ given by
$$
\pi_{8}(\tau)=\frac{\eta(\tau)^{4}\eta(4\tau)^{2}}{\eta(2\tau)^{2}\eta(8\tau)^{4}}.
$$
  Then
$$
\frac{\eta(z_{1})^{4}\eta(4z_{1})^{2}}{\eta(2z_{1})^{2}\eta(8z_{1})^{4}}-\frac{\eta(z_{2})^{4}\eta(4z_{2})^{2}}{\eta(2z_{2})^{2}\eta(8z_{2})^{4}}=\left(q_{1}^{-1}-q_{2}^{-1}\right)\prod_{m,n>0}\prod_{d|8}\left(1-\left(q_{1}^{m}q_{2}^{n}\right)^{\frac{8}{d}}\right)^{A(mn,d)}.
$$
where
$$
\sum_{\ell=-1}^{\infty}A(\ell,8)q^{\ell}=\frac{\eta(\tau)^{4}\eta(4\tau)^{2}}{\eta(2\tau)^{2}\eta(8\tau)^{4}}+4,
$$
$$
\sum_{\ell=1}^{\infty}A(\ell,4)q^{\ell}=4-4\frac{\eta(\tau)^{4}\eta(8\tau)^{4}\eta(2\tau)^{2}}{\eta(4\tau)^{10}},
$$
$$
A(\ell,2)=a(2\ell)\quad\mbox{and}\quad \sum_{\ell=1}^{\infty}a(\ell)q^{\ell}=8-8e\left(1/12\right)\frac{\eta(2\tau)^{2}\eta(4\tau)^{4}}{\eta(8\tau)^{2}\eta\left(\tau+\frac{1}{4}\right)^{4}}=8-8\prod_{n=1}^{\infty}\frac{(1-q^{2n})^{2}(1-q^{4n})^{4}}{(1-q^{8n})^{2}(1-(iq)^{n})^{4}},
$$
$$
A(\ell,1)=b(8\ell)\quad\mbox{and}\quad \sum_{\ell=1}^{\infty}b(\ell)q^{\ell}=32\frac{\eta(8\tau)^{4}\eta(2\tau)^{2}}{\eta(4\tau)^{2}\eta(\tau)^{4}}.
$$
\end{Example}

\begin{Remark}
Similar formulas for $\Gamma_{0}^{+}(N)$ with $N$ square free have first been worked out in \cite{bor5} by Borcherds, where $\Gamma_{0}^{+}(N)$ is the discrete subgroup of ${\rm SL}_{2}(\mathbb{R})$ generated by $\Gamma_{0}(N)$ and all of its Atkin--Lehner involutions.  In recent work \cite{can}, Carnahan obtains similar
formulas for completely replicable modular functions.
\end{Remark}

Following from Corollary \ref{diffpiexp}, we can relate the canonical basis elements in $z_{2}$ of the space of weakly holomorphic modular functions for $\Gamma_{0}(N)$ with pole supported only at $i\infty$ to the logarithmic derivative of $\pi_{N}(z_{1})-\pi_{N}(z_{2})$ with respect to $z_{1}$ as follows.
\begin{Corollary}
\label{PN}
Let $\pi_{N}(z)$ be a Hauptmodul for a genus zero $\Gamma_{0}(N)$. Then
$$
-\frac{1}{2\pi i}\frac{\pi_{N}'(z_{1})}{\pi_{N}(z_{1})-\pi_{N}(z_{2})}=1+\sum_{n=1}^{\infty}P_{N,n}(\pi_{N}(z_{2}))q_{1}^{n}
$$
for some polynomial $P_{N,n}$ of degree $n$ and ${\rm Im}(z_{1})\gg{\rm Im}(z_{2})$. Then $P_{N,n}(\pi_{N}(z_{2}))=q_{2}^{-n}+O(q_{2})$.
\end{Corollary}

\begin{Remark}
Corollary~\ref{PN} for cases $N\in\{2,3,5,7,13\}$ had been given in \cite{A}, and has recently been extended to arbitrary genus zero Fuchsian subgroups of ${\rm SL}_{2}(\mathbb{R})$ in \cite{Y19} by the author.
\end{Remark}

The second main results of this work are the Gross--Zagier type CM value formulas for $\pi_{p}(\tau)$ with $p\in\{3,5,7,13\}$ in the following, which, loosely speaking, follows from the so-called big CM value formula of Bruinier, Kudla and Yang \cite{bky}, and Theorem~\ref{borthm}.

{\begin{Theorem}
\label{gztype}
Let $E_{i}=\mathbb{Q}(\sqrt{d_{i}})$ be two imaginary quadratic fields of fundamental discriminants $d_{i}$ with $(d_{1},d_{2})=1$. Let $F=\mathbb{Q}(\sqrt{D})$ with $D=d_{1}d_{2}$ and $E=E_{1}E_{2}=\mathbb{Q}(\sqrt{d_{1}},\sqrt{d_{2}})$. Let $\pi_{p}(\tau)$ be a Hauptmodul of $\Gamma_{0}(p)$. Then
\begin{align}
&\sum_{([\mathfrak{a}_{1}],[\mathfrak{a}_{2}])\in S(p,d_{1},d_{2})}\log|\pi_{p}(\tau_{\mathfrak{a}_{1}}/p)-\pi_{p}(\tau_{\mathfrak{a}_{2}}/p)|\nonumber\\
\label{gzpi}
&=-\frac{|S(p,d_{1},d_{2})|w_{1}w_{2}}{32h(E_{1})h(E_{2})}\left(\sum_{\substack{t=\frac{2m+D+\sqrt{D}}{2}\\ |2m+D|<\sqrt{D}\\m\in\Z}}a\left(\frac{t}{\sqrt{D}},\phi_{0,0}\right)+\frac{24}{p-1}\sum_{k=1}^{p-1}a_{0}(\phi_{0,k})\right)
\end{align}
where 
$$
S(p,d_{1},d_{2}):=\left\{([\mathfrak{a}_{1}],[\mathfrak{a}_{2}])\in {\rm Cl}_{p}(E_{1})\times {\rm Cl}_{p}(E_{2})|\,\text{$[\mathfrak{a}_{i}]$ are representatives of ${\rm Cl}_{p}(E_{i})$ with ${\rm N}(\mathfrak{a}_{1})={\rm N}(\mathfrak{a}_{2})$}\right\},
$$
${\rm Cl}_{p}(E_{i})$ denote the ring class group of conductor $p$ of $E_{i}$, $w_{i}$ is the number of units of $E_{i}$, $\tau_{\mathfrak{a}}=p\frac{b+\sqrt{d}}{2a}$ is the CM point associated to the integral ideal representative $\mathfrak{a}=[a,p\frac{b+\sqrt{d}}{2}]$ of ${\rm Cl}_{p}(E_{i})$, and $a\left(\frac{t}{\sqrt{D}},\phi_{0,0}\right)$ and $a_{0}(\phi_{0,k})$ are computed and expressed explicitly in Section~6. In particular, when $\left(\frac{d_{1}}{p}\right)=\left(\frac{d_{2}}{p}\right)=1$ and $p\mathcal{O}_{F}=\mathfrak{p}_{1}\mathfrak{p}_{2}$, one has
$$
a\left(\frac{t}{\sqrt{D}},\phi_{0,0}\right)=-4{\frac{p^{2}}{(p-1)^{2}}}\sum_{i=0}^{p-1}\sum_{\text{$\mathfrak{p}$ inert in $E/F$}}\frac{1+{\rm ord}_{\mathfrak{p}}(t)}{2}\log(N(\mathfrak{p}))\prod_{\mathfrak{q}\nmid p}\rho_{\mathfrak{q}}(t\mathfrak{p}^{-1})\prod_{j=1}^{2}\frac{W_{t,\mathfrak{p}_{j}}^{\psi_{F}'}(0,\phi_{\mathfrak{p}_{j}}^{(i)})}{\gamma(W_{\mathfrak{p}_{j}}')},
$$
where the values of $\frac{W_{t,\mathfrak{p}_{j}}^{\psi_{F}'}(0,\phi_{\mathfrak{p}_{j}}^{(i)})}{\gamma(W_{\mathfrak{p}_{j}}')}$ are given in Corollary~\ref{evalw}, and
$$
a_{0}(\phi_{0,k})=-h(E_{1})h(E_{2})\frac{4\log{p}}{p-1}.
$$

Moreover, the left hand side of \eqref{gzpi} can be reformulated in the language of quadratic forms
$$
\sum_{([\mathfrak{a}_{1}],[\mathfrak{a}_{2}])\in S(p,d_{1},d_{2})}\log|\pi_{p}(\tau_{\mathfrak{a}_{1}}/p)-\pi_{p}(\tau_{\mathfrak{a}_{2}}/p)|=\sum_{(Q_{1},Q_{2})\in S_{\mathcal{Q}}(p,d_{1},d_{2})}\log|\pi_{p}(\tau_{Q_{1}})-\pi_{p}(\tau_{Q_{2}})|,
$$
where 
\begin{align*}
&S_{\mathcal{Q}}(p,d_{1},d_{2})\\
&:=\left\{\left.([Q_{1}],[Q_{2}])\in\mathcal{Q}_{d_{1}}(p)/\Gamma_{0}(p)\times\mathcal{Q}_{d_{2}}(p)/\Gamma_{0}(p)\right|\, \mbox{$[Q_{i}]$ are representatives with $Q_{1}(1,0)=Q_{2}(1,0)$}\right\},
\end{align*}
$\mathcal{Q}_{d}(p)$ denotes the set of primitive and positive definite binary quadratic forms $aX^{2}+bXY+cY^{2}$ of discriminant $d$ with $(a,p)=1$, and $\tau_{Q}$ is the unique CM point in $\mathbb{H}$ given by $Q(\tau,1)=0$ for a quadratic form $Q(X,Y)$.
\end{Theorem}
}

\begin{Remark}
The way we compute the Gross--Zagier type CM value formulas above in this work was first initiated in \cite{yangyin} in which Yang and Yin prove a Gross--Zagier type CM value formula for $\pi_{2}(\tau)$ first conjectured by Yui and Zagier \cite{yz}. In general, one can similarly obtain Gross--Zagier type CM value formulas for all $\pi_{N}(\tau)$ by carefully computing the constant terms of the corresponding theta-lift input and relevant local Whittaker functions (see Lemma~\ref{lemFN} and Subsection~\ref{bigcmfor}). We refer the reader to \cite{yangyin,yyy} for comprehensive descriptions of these computations, and leave the details to the reader.
\end{Remark}

\begin{Remark}
Similar to the remarks given on \cite[p. 352]{KRY}, one may note that the first term inside the parenthesis of the right hand side of \eqref{gzpi} can be interpreted as the $D$-th Fourier coefficient of the product of theta function $\theta(\tau)=\sum_{n\in\mathbb{Z}}q^{n^{2}}$ and the derivative of a non-holomorphic weight~1 Hilbert Eisenstein series $E^{*'}(4\tau^{\Delta},0,\phi_{0,0})$ at diagonal (see Section~5 for definitions).
\end{Remark}

\begin{Example}
\label{exam1}
Taking $p=3$, $d_{1}=-8$ and $d_{2}=-11$, one can first check using the identifications with quadratic forms  that $S(3,-8,-11)={\rm Cl}_{3}(E_{1})\times {\rm Cl}_{3}(E_{2})$. Since $D=88$ and $\left(\frac{-8}{3}\right)=\left(\frac{-11}{3}\right)=1$, then one can check that the $t$'s that we need to consider are 
$$
\frac{\pm8+\sqrt{88}}{2},\quad\frac{\pm6+\sqrt{88}}{2},\quad \frac{\pm4+\sqrt{88}}{2},\quad \frac{\pm2+\sqrt{88}}{2},\quad \mbox{and}\quad\frac{\sqrt{88}}{2},
$$
and thus, using the given formula together with Corollary~\ref{evalw}, one has that
\begin{align*}
&a\left(\frac{8+\sqrt{88}}{2\sqrt{88}},\phi_{0,0}\right)=a\left(\frac{6+\sqrt{88}}{2\sqrt{88}},\phi_{0,0}\right)=a\left(\frac{2+\sqrt{88}}{2\sqrt{88}},\phi_{0,0}\right)\\
&=a\left(\frac{\sqrt{88}}{2\sqrt{88}},\phi_{0,0}\right)=a\left(\frac{-6+\sqrt{88}}{2\sqrt{88}},\phi_{0,0}\right)=0,
\end{align*}
$$
a\left(\frac{4+\sqrt{88}}{2\sqrt{88}},\phi_{0,0}\right)=a\left(\frac{-4+\sqrt{88}}{2\sqrt{88}},\phi_{0,0}\right)=a\left(\frac{-8+\sqrt{88}}{2\sqrt{88}},\phi_{0,0}\right)=-4\log2,
$$
$$
a\left(\frac{-2+\sqrt{88}}{2\sqrt{88}},\phi_{0,0}\right)=-4\log7, \quad a_{0}(\phi_{0,k})=-2\log3,\quad\mbox{and}\quad -\frac{|S(p,d_{1},d_{2})|w_{1}w_{2}}{32h(E_{1})h(E_{2})}=-\frac{1}{2}.
$$
Incorporating all of these, one can obtain the following prime factorization for the product of CM values
$$
\prod_{([\mathfrak{a}_{1}],[\mathfrak{a}_{2}])\in {\rm Cl}_{3}(E_{1})\times {\rm Cl}_{3}(E_{2})}|\pi_{3}(\tau_{\mathfrak{a}_{1}}/3)-\pi_{3}(\tau_{\mathfrak{a}_{2}}/3)|=2^{6}3^{24}7^{2}.
$$
\end{Example}


\begin{Remark}
We remark that formulas analogous to \eqref{gzpi} have been recently established by a different method by the author in \cite{Y20}. Also, Gross--Zagier type CM value formulas for the genus zero Fricke subgroups $\Gamma_{0}^{+}(p)$ have been recently derived in another work of the author \cite{Y} using the so called small CM value formula \cite{S}. We refer the reader to \cite{bky} for a brief explanation on the distinction between the concepts related to ``big CM'' and ``small CM'' and to \cite{Y} for a brief explanation on why the small CM value formula may not work in our case.

\end{Remark}

This work is divided into two parts and is organized as follows. In the first part, we briefly review the  theory of Borcherds lifts, realize a family of Shimura varieties as degenerate Hilbert modular surfaces, set up several preliminary results and construct the desired weakly holomorhic modular form $F_{N}$ for the Weil representation. Proofs of Theorem \ref{borthm}, Corollaries \ref{diffpiexp} and \ref{PN} are given in Section~4. In the second part of this work, we will briefly review the concepts of big CM cycles and big CM value formula and show how to employ such formula together with Theorem \ref{borthm} to obtain Theorem \ref{gztype} in Section~5. Computations of  $a(\frac{t}{\sqrt{D}},\phi_{0,0})$ and $a_{0}(\phi_{0,k})$ mentioned in Theorem~\ref{gztype} are carried out in Section~6.


\part{Difference of a Hauptmodul for $\Gamma_{0}(N)$}
\section{Review of Borcherds Lifts}
\label{borlif}

In this section, we briefly review the theory of Borcherds lifts in the adelic setting and relevant concepts such as Shimura varieties and special divisors, and we also realize a family of Shimura varieties as degenerate Hilbert modular surfaces on which the difference of Hauptmoduls are defined. We rely heavily on \cite{K} (also see \cite{yangyin}).

\subsection{Shimura Variety of Type~$(n,2)$} For a positive integer $n$, let $V$ be a rational quadratic space with a quadratic form $Q(\cdot)$ of signature $(n,2)$ and associated bilinear form $(\cdot,\cdot)$. 
Let $H=\mbox{GSpin}(V)$ be the general spin group of $V$, then there is an exact sequence
$$
1\rightarrow\mathbb{G}_{m}\rightarrow H\rightarrow \mbox{SO}(V)\rightarrow 1.
$$
  For a $\mathbb{Q}$-algebra $F$, we write  $V_{F}$ for $V\otimes_{\Q}F$. A Hermitian symmetric domain for $H(\RR)$ is the Grassmannian of oriented negative~2-planes of $V_{\RR}$, denoted by $\mathbb{D}$.
Denote by $\mathbb{A}_{\mathcal{K}}$ the adele ring over a number field $\mathcal{K}$ and by $\mathbb{A}_{\mathcal{K},f}$ the associated finite adele ring.  For a compact open subgroup $K\subset H(\A_{\Q,f})$, there is an associated (open) Shimura variety $\mathcal{X}_{K}$ over $\Q$ such that
   $$\mathcal{X}_{K}(\C)=H(\Q)\backslash\left(\D\times H(\A_{\mathbb{Q},f})/K\right).$$
    Let
  $$
  \mathcal{L}=\{Z\in V_{\C}|\,\mbox{$(Z,Z)=0$ and $(Z,\bar{Z})<0$}\}.
  $$
Then we can see that $\D$ possesses a complex structure via the isomorphism $pr:\mathcal{L}/\C^{*}\to\D$ sending $Z=X+iY$ to $\RR X+\RR(-Y)$. There is another useful realization for $\D$ as follows. Take two isotropic (zero-norm) elements $\ell$ and $\ell'$ of $V$ with $(\ell,\ell')=1$, and let $V_{0}=(\mathbb{Q}\ell+\mathbb{Q}\ell')^{\perp}$ be the orthogonal complement of the plane spanned by $\ell$ and $\ell'$.
 Then we define a so-called tube domain associated to $\ell$ and $\ell'$ by
$$
\mathcal{H}=\{Z=X+iY\in V_{0,\C}|\,\text{$X,Y\in V_{0,\RR}$ and $Q(Y)<0$}\}
$$
which is isomorphic to $\mathcal{L}/\C^{*}$ via $w(Z)=\ell'-Q(Z)\ell+Z$. 
Then the map $w$ induces an action of $\Gamma=K\cap H(\Q)^{+}$ on $\mathcal{H}$, where $H(\Q)^{+}=H(\Q)\cap H(\RR)^{+}$ and $H(\RR)^{+}$ is the identity component of $H(\RR)$, and induces an automorphy factor $j(g,Z)$ characterized by the following identity
$$
g\cdot w(Z)=\nu(g)j(g,Z)w(g\cdot Z),
$$ 
where $\nu(g)$ is the spinor norm of $g$. Note that this action preserves the two connected components $\mathcal{H}^{\pm}$ of $\mathcal{H}$. Fix one of these two connected components, say, $\mathcal{H}^{+}$. Assuming for simplicity that $H(\mathbb{A}_{\mathbb{Q}})=H(\mathbb{Q})H(\mathbb{R})^{+}K$, which is guaranteed by an appropriate choice of $K$ and the strong approximation theorem (see, e.g., \cite[Ch. 7]{PR}), we have the identification $\mathcal{X}_{K}\cong \Gamma\backslash\mathcal{H}^{+}$.

\begin{Definition}
A meromorphic modular form on $\mathcal{H}^{+}$ of weight $k$ for $\Gamma$ is a meromorphic function $f:\mathcal{H}^{+}\to\mathbb{C}$ such that
$
f(\gamma\cdot Z)=j(\gamma,Z)^{k}f(Z)
$
for all $\gamma\in\Gamma$. Moreover, under the identification given above, such a function $f$ can be  equivalently viewed as a meromorphic modular form defined on $\mathbb{D}\times H(\mathbb{A}_{\Q,f})$ satisfying 
\begin{enumerate}[(1)]
\item{$f(Z,hk)=f(Z,h)$ for all $k\in K$,}
\item{$f(\gamma\cdot{Z},\gamma{h})=j(\gamma,z)^{k}f(Z,h)$ for all $\gamma\in H(\mathbb{Q})$.}
\end{enumerate}
\end{Definition}

\subsection{Special Divisor}\label{specialdiv} For a vector $X\in V$ with $Q(X)>0$ and $h\in H(\mathbb{A}_{\Q,f})$, let
$$
H_{X}=\{g\in H|\,g\cdot X=X\},\quad\mathbb{D}_{X}=\{Z\in\mathbb{D}|\,(Z,X)=0\},\quad\mbox{and}\quad K_{X,h}=H_{X}(\mathbb{A}_{\Q,f})\cap hKh^{-1}.
$$
Then the map
$$
H_{X}(\Q)\backslash\left(\mathbb{D}_{X}\times H_{X}(\mathbb{A}_{\Q,f})/K_{X,h}\right)\to \mathcal{X}_{K}(\mathbb{C}),\quad [Z,g]\to[Z,gh]
$$
gives a divisor $Z(X,h)$ in $\mathcal{X}_{K}$ which is defined over $\Q$. For a positive rational number $m$ and $\varphi\in\mathcal{S}(V_{\mathbb{A}_{\Q,f}})^{K}$, the $K$-invariant subspace of the  Schwartz--Bruhat space of $V_{\mathbb{A}_{\Q,f}}$ consisting of compactly supported and locally constant functions on $V(\mathbb{A}_{\Q,f})$, if there is a $X\in V$ with $Q(X)=m$, the special divisor of index $(m,\varphi)$ in $\mathcal{X}_{K}$ is defined by
$$
Z(m,\varphi)=\sum_{h\in H_{X}(\Q)\backslash H(\mathbb{A}_{\Q,f})/K}\varphi(h^{-1}\cdot X)Z(X,h).
$$
When there is no such $X$, we simply set $Z(m,\varphi)=0$. One can check that this definition does not depend on the choice of $X$ and $h$ (see e.g., \cite[Lemma 4.2.1]{Z}). Moreover, $Z(m,\varphi)$ can be described explicitly  \cite[Proposition 5.4]{Ku} as a divisor of $\Gamma\backslash\mathcal{H}^{+}$ as follows:
$$
Z(m,\varphi)=\sum_{X\in\Gamma\backslash\Omega_{m}(\mathbb{Q})}\varphi(X){\rm pr}(\mathbb{D}_{X})
$$
where $\Omega_{m}(\mathbb{Q})=\{X\in V|\,Q(X)=m\}$ and ${\rm pr}:\mathcal{H}^{+}\to\Gamma\backslash\mathcal{H}^{+}$ is the natural projection.

\subsection{Theta-lift and Borcherds Theorem}\label{thetalift}
For $Z\in\mathbb{D}$, let ${\rm pr}_{Z}:V_{\mathbb{R}}\to Z$ be the projection map, and for $X\in V_{\mathbb{R}}$, let $R(X,Z)=-({\rm pr}_{Z}(X),{\rm pr}_{Z}(X))$. Then we define 
$$
(X,X)_{Z}=(X,X)+2R(X,Z),
$$
and our Gaussian for $V$ is the function
$$
\varphi_{\infty}(X,Z)=e^{-\pi(X,X)_{Z}}.
$$
For $\tau\in\mathbb{H}$ with $\tau=u+iv$, let
$$
g_{\tau}=\begin{pmatrix}1&u\\0&1\end{pmatrix}\begin{pmatrix}v^{\frac{1}{2}}&0\\0&v^{-\frac{1}{2}}\end{pmatrix},
$$
and $g_{\tau}'=(g_{\tau},1)\in {\rm Mp}_{2}(\mathbb{R})$, the metaplectic group. Let  $G={\rm SL}_{2}$ and $\omega$ be the Weil representation of  $G(\mathbb{A}_{\Q})$ on $\mathcal{S}(V_{\mathbb{A}_{\Q,f}})$. Then for the linear action of $H(\mathbb{A}_{\Q,f})$ we write $\omega(h)\varphi(X)=\omega(h^{-1}\cdot X)$ for $\varphi\in\mathcal{S}(V_{\mathbb{A}_{\Q,f}})$. For $Z\in\mathbb{D}$ and $h\in H(\mathbb{A}_{\Q,f})$, we have the linear functional on $\mathcal{S}(V_{\mathbb{A}_{\Q,f}})$ given by
\begin{equation}
\label{sietheta}
\varphi\to\theta(\tau,Z,h;\varphi):=v^{-\frac{n}{4}+\frac{1}{2}}\sum_{X\in V_{\mathbb{Q}}}\omega(g_{\tau}')\left(\varphi_{\infty}(\cdot,Z)\otimes\omega(h)\varphi\right)(X).
\end{equation}

Let $L$ be an even lattice of $V$, i.e., $Q(X)\in\mathbb{Z}$ for all $X\in L$, and let $L'$ be the dual lattice of $L$ defined by
$$
L'=\{X\in V|\,(X,L)\subset\mathbb{Z}\}.
$$
Embed ${\rm SL}_{2}(\mathbb{Z})$ into ${\rm SL}_{2}(\hat{\mathbb{Z}})$ diagonally, and let $\mathcal{S}_{L}$ be the subspace of $\mathcal{S}(V_{\mathbb{A}_{\Q,f}})$ consisting of functions with support in $\hat{L}'$ and constant on cosets of $\hat{L}$, where $\hat{L}=L\otimes_{\mathbb{Z}}\hat{\mathbb{Z}}$. Then 
$$
\mathcal{S}_{L}=\bigoplus_{\eta\in L'/L}\mathbb{C}\phi_{\eta},\quad \phi_{\eta}={\rm Char}(\eta+\hat{L}),
$$
where ${\rm Char}(\cdot)$ denotes the characteristic function associated to a certain set. One can check that $\mathcal{S}_{L}$ is ${\rm SL}_{2}(\Z)$-invariant under the Weil representation $\omega$, and we denote this representation by $\omega_{L}$. Explicitly, one has
\begin{align*}
\omega_{L}(T)\phi_{\mu}&=e(-Q(\mu))\phi_{\mu},\\
\omega_{L}(S)\phi_{\mu}&=\frac{e((n-2)/8)}{\sqrt{|L'/L|}}\sum_{\gamma\in L'/L}e((\gamma,\mu))\phi_{\gamma}
\end{align*}
where $e(\tau):=\exp(2\pi i\tau)$, $T=\begin{pmatrix}1&1\\0&1\end{pmatrix}$ and $S=\begin{pmatrix}0&-1\\1&0\end{pmatrix}$.
\begin{Definition}
A holomorphic function ${F}:\mathbb{H}\to \mathcal{S}_{L}$ is a weakly holomorphic modular form of integral weight $k$ for the Weil representation $\omega_{L}$ if
\begin{enumerate}[(i)]
\item{${F}(\gamma\tau)=(c\tau+d)^{k}\omega_{L}(\gamma){F}(\tau)$ for all $\gamma=\begin{pmatrix}a&b\\c&d\end{pmatrix}\in{\rm SL}_{2}(\mathbb{Z})$,}
\item{${F}(\tau)$ has a Fourier expansion
$$
{F}(\tau)=\sum_{\eta\in L'/L}\sum_{\substack{m\in -Q(\eta)+\mathbb{Z}\\m\gg-\infty}}c(m,\eta)q^{m}\phi_{\eta}
$$
where the condition $m\equiv -Q(\eta)\pmod{\mathbb{Z}}$ follows from the transformation law for $T$.
}
\end{enumerate}
Furthermore, denote by $M_{k,\omega_{L}}^{!}$ the space of weakly holomophic modular forms of weight $k$ for the Weil representation $\omega_{L}$.
\end{Definition}
For the theta function called Siegel theta function (see \eqref{sietheta} for the definition of the summands)
$$
\theta(\tau,Z,h)=\sum_{\mu\in L'/L}\theta(\tau,Z,h;\phi_{\mu}),
$$
which is indeed of weight~$\frac{n}{2}-1$, we can pair it with the weight~$1-\frac{n}{2}$ modular form ${F}(\tau)$ by the following  $\mathbb{C}$-bilinear pairing
$$
\langle {F}(\tau),\theta(\tau,Z,h)\rangle=\sum_{\mu\in L'/L}\sum_{m\in Q(\mu)+\mathbb{Z}}c(m,\mu)q^{m}\theta(\tau,Z,h;\phi_{\mu}).
$$
Using this pairing, we define a regularized integral as in \cite{bor}, called theta-lift,
$$
\Phi(Z,h;{F}):=\underset{s=0}{{\rm CT}}\left\{\lim_{t\to\infty}\int_{\mathcal{F}_{t}}\langle {F}(\tau),\theta(\tau,Z,h)\rangle v^{-2-s}dudv\right\}
$$
where $\underset{s=0}{{\rm CT}}$ denotes the constant term in the Laurent expansion at $s=0$ of 
$$
\lim_{t\to\infty}\int_{\mathcal{F}_{t}}\langle {F}(\tau),\theta(\tau,Z,h)\rangle v^{-2-s}dudv,
$$
$\mathcal{F}_{t}$ is the truncated fundamental domain defined by
$$
\mathcal{F}_{t}:=\{\tau\in\mathcal{F}|\, {\rm Im}(\tau)\leq t\}
$$
and $\mathcal{F}$ is the usual fundamental domain for the action of ${\rm SL}_{2}(\mathbb{Z})$ on $\mathbb{H}$.

Now we are ready to state the celebrated Borcherds Theorem. Under the assumption $H(\mathbb{A}_{\mathbb{Q}})=H(\mathbb{Q})H(\mathbb{R})^{+}K$, we can simply set $h=1$ and omit it for the simplicity of notation.
\begin{Theorem}[Borcherds]
\label{bors}
For a given $f\in M_{1-\frac{n}{2},\omega_{L}}^{!}$ with
$$
f(\tau)=\sum_{\mu\in L'/L}\sum_{\substack{m\in-Q(\mu)+\Z\\m\gg-\infty}}c(m,\mu)q^{m}\phi_{\mu},
$$
and  $c(m,\mu)\in\Z$ for $m<0$,  and  assuming $\Gamma$ is the stabilizer subgroup of $f$ as an element of $\mathbb{C}[L'/L]$, there is a meromorphic modular form $\Psi(Z,f)$ called Borcherds lift on $\mathbb{D}\times H(\mathbb{A}_{\Q,f})$ (or say  on $\mathcal{H}^{+}$) of weight $c(0,0)/2$ for $\Gamma$ such that
\begin{enumerate}
\item{
the divisor of $\Psi(Z,f)$ on $\mathcal{X}_{K}$ is given by
$$
{\rm div}(\Psi(Z,f)^2)=\sum_{\mu\in L'/L}\sum_{\substack{m\in Q(\mu)+\Z\\0<m}}c(-m,\mu)Z(m,\mu),
$$
where $Z(m,\mu)$ is the special divisor of index $(m,\mu)$ as defined in Subsection~\ref{specialdiv}, 
}
\item{the following relation 
$$
\Phi(Z,f)=-4\log|\Psi(Z,f)|-c(0,0)\left(2\log|Y|+\Gamma'(1)+\log(2\pi)\right)
$$
holds,}
\item{
near each cusp $\Q\ell$ of $\mathcal{X}_{K}$, the meromorphic function $\Psi(Z,f)$ has a product expansion called Borcherds product of the form
$$
\Psi(Z,f)=Ce\left((Z,\rho(W_{f,\ell_{M}},f)\right)\prod_{\substack{\lambda\in M_{\ell}'\\(\lambda,W_{f,\ell_{M}})>0}}\prod_{\substack{\mu\in L_{\ell}'/L\\ p(\mu)=\lambda+M_{\ell}}}\left[1-e\left((\lambda,Z)+(\mu,\ell')\right)\right]^{c(-Q(\mu),\mu)}
$$
where $C$ is a constant with absolute value
$$
\left|\prod_{\substack{\delta\in\Z/N\\\delta\ne0}}\left(1-e(\delta/N)\right)^{\frac{c(0,\frac{\delta}{N}\ell)}{2}}\right|.
$$
 For the definitions of the notation, we refer the reader to \cite[Subsection 2.1]{yangyin}. 
}
\end{enumerate}
\end{Theorem}


\subsection{A Shimura Variety as a degenerate Hilbert modular surface}\label{hil} In this subsection, we will see how to realize a family of  Shimura varieties $\mathcal{X}_{K_{N}}$ of type $(2,2)$ for some open compact subgroup $K_{N}$ as degenerate Hilbert modular surfaces $Y_{0}(N)\times Y_{0}(N)$.

Let $V=M_{2}(\mathbb{Q})$ be a rational quadratic space with the quadratic form $Q(\cdot):=\det(\cdot)$ of signature $(2,2)$. Then the general spin group $H={\rm GSpin}(V)$ of $V$ is
$$
H=\{(g_{1},g_{2})\in{\rm GL}_{2}\times {\rm GL}_{2}|\, \det g_{1}=\det g_{2}\}
$$
and acts on $V$ via $(g_{1},g_{2})\cdot X=g_{1}Xg_{2}^{-1}$. Taking $\ell=\begin{pmatrix}0&-1\\0&0\end{pmatrix}$ and $\ell'=\begin{pmatrix}0&0\\1&0\end{pmatrix}$, one can check that 
$$
\mathcal{H}=\left\{\left.\begin{pmatrix}z_{1}&0\\0&-z_{2}\end{pmatrix}\right|\,\Im(z_{1})\Im(z_{2})>0\right\}.
$$
The following proposition is useful and well known (see, e.g., \cite[Proposition~3.1]{yangyin}).
\begin{Proposition}
\label{yyprop}
Define
$$
\tilde{w}:\HH^{2}\cup(\HH^{-})^{2}\to\mathcal{L}
$$
by $\tilde{w}((z_{1},z_{2}))=\begin{pmatrix}z_{1}&-z_{1}z_{2}\\1&-z_{2}\end{pmatrix}$. Then the composition $pr\circ\tilde{w}$ gives an isomorphism between $\HH^{2}\cup(\HH^{-})^{2}$ and $\D$. One can check that such an isomorphism induces an action of $H(\RR)$ on $\HH^{2}\cup(\HH^{-})^{2}$ via the usual fractional linear transformation, i.e.,
$$
(g_{1},g_{2})\cdot(z_{1},z_{2})=(g_{1}\cdot z_{1},\,g_{2}\cdot z_{2}),
$$
and an automorphy factor $j(g_{1},g_{2};z_{1},z_{2})=(c_{1}z_{1}+d_{1})(c_{2}z_{2}+d_{2})$ for 
$g_{j}=\begin{pmatrix}a_{j}&b_{j}\\c_{j}&d_{j}\end{pmatrix}$.
\end{Proposition}
Now let 
$$
K_{0}(N)=\left\{\left.g\in\GL_{2}(\hat{\Z})\right|\,g\equiv\begin{pmatrix}*&*\\0&*\end{pmatrix}\pmod{N}\right\}
$$
and denote $K=\left(K_{0}(N)\times K_{0}(N)\right)\cap H(\A_{f})$ by $K_{N}$. Then $\Gamma=H(\mathbb{Q})^{+}\cap K_{N}=\Gamma_{0}(N)\times\Gamma_{0}(N)$ denoted by $\Gamma_{N}$. The strong approximation theorem tells that $H(\mathbb{A}_{\Q})=H(\Q)H(\mathbb{R})^{+}K_{N}$, and  together with Proposition \ref{yyprop} it implies the realization of $\mathcal{X}_{K_{N}}\cong Y_{0}(N)\times Y_{0}(N)$ as a degenerate Hilbert modular surface. Moreover, set $L=\begin{pmatrix}\Z&\Z\\N\Z&\Z\end{pmatrix}$ and denote it by $L_{N}$. Then one can check that $L_{N}'=\begin{pmatrix}\Z&\Z/N\\\Z&\Z\end{pmatrix}$ and the special  divisor of index $(n,\mu)$ can be expressed as a divisor on $Y_{0}(N)\times Y_{0}(N)$
$$
Z(n,\mu)=\left(\Gamma_{0}(N)\times\Gamma_{0}(N)\right)\backslash\left\{z=(z_{1},z_{2})\in\HH^{2}|\,\text{$\tilde{w}(z)\perp X$ for some $X\in\mu+L_{N}$ with $Q(\mu)=n$}\right\}.
$$
Finally, one can check that $\Gamma_{N}\subset {\rm Aut}(L_{N}'/L_{N})\cap {\rm Aut}(L_{N})$, where ${\rm Aut}(L_{N}'/L_{N})$ and ${\rm Aut}(L_{N})$ denote the automorphism groups of $L_{N}'/L_{N}$ and $L_{N}$, respectively. 

\section{Preliminary Results}
In this section, we aim to construct an appropriate Borcherds lift input $F_{N}$ for $\pi_{N}(z_{1})-\pi_{N}(z_{2})$. First, with $L_{N}=\begin{pmatrix}\Z&\Z\\N\Z&\Z\end{pmatrix}$ and $L_{N}'=\begin{pmatrix}\Z&\Z/N\\\Z&\Z\end{pmatrix}$, we note that $\left\{\mu\right\}_{\mu\in L_{N}'/L_{N}}=\left\{\begin{pmatrix}0&j/N\\k&0\end{pmatrix}\right\}_{0\leq j,k\leq N-1}$. Then to simplify our notation, we write $\mu_{j,k}$ for $\begin{pmatrix}0&j/N\\k&0\end{pmatrix}$, and write $L_{N}'/L_{N}=\left\{\mu_{j,k}\right\}_{0\leq j,k\leq N-1}$. We also write $\phi_{j,k}$ for $\phi_{\mu_{j,k}}\in \mathcal{S}_{L_{N}}$.

\begin{Lemma}
\label{zerolem}
Let $Z(1,\mu_{0,0})$ be the special divisor of index $(1,\mu_{0,0})$. Then
\begin{align*}
Z(1,\mu_{0,0})&=\left\{\left.(\tau,\tau)\right|\tau\in Y_{0}(N)\right\}.
\end{align*}
\end{Lemma}

\begin{proof}
By the identification given in Subsection~\ref{hil}, we have
\begin{align*}
Z(1,\mu_{0,0})&=(\Gamma_{0}(N)\times\Gamma_{0}(N))\backslash\left\{(z_{1},z_{2})\in\mathbb{H}^{2}|\,\text{$\tilde{w}(z_{1},z_{2})\perp X$ for some $X\in\mu_{0,0}+L_{N}$ with $Q(X)=1$}\right\}\\
&=(\Gamma_{0}(N)\times\Gamma_{0}(N))\backslash\left\{(z_{1},z_{2})\in\mathbb{H}^{2}|\,\text{$z_{2}=\frac{az_{1}+b}{cz_{1}+d}$ for some $\begin{pmatrix}a&b\\c&d\end{pmatrix}\in\Gamma_{0}(N)$}\right\}\\
&=\left\{\left.(\tau,\tau)\right|\tau\in Y_{0}(N)\right\}.
\end{align*}
\end{proof}
Note that the divisor of $\pi_{N}(z_{1})-\pi_{N}(z_{2})$ is 
$
\left\{\left.(\tau,\tau)\right|\tau\in Y_{0}(N)\right\}.
$
Then Lemma~\ref{zerolem} tells us that if $F_{N}$ is a Borcherds lift input for $\pi_{N}(z_{1})-\pi_{N}(z_{2})$, then ${\rm div}(\Psi(z,F_{N}))$ must be $Z(1,\mu_{0,0})$, and thus the $\phi_{0,0}$-component function of $F_{N}$ must have a simple pole at $i\infty$ of residue~1, and the other component functions are all holomorphic at $i\infty$.

\begin{Lemma}
\label{wt0}
For $N\ne4$, the following set
$$
\bigcup_{\substack{d|N}}\left\{\sum_{k=0}^{d-1}\sum_{j=0}^{\frac{N}{d}-1}\phi_{jd,\,k\frac{N}{d}}\right\}
$$
is a basis for the subspace of $M_{0,\rho_{L_{N}}}^{!}$ consisting of elements whose $\phi_{j,k}$-component functions are all constants.
 For $N=4$, the following set
  $$
\begin{Bmatrix}\phi_{0,0}+\phi_{1,0}+\phi_{2,0}+\phi_{3,0},\\
\phi_{0,0}+\phi_{0,1}+\phi_{0,2}+\phi_{0,3},\\
\phi_{2,0}-\phi_{0,1}+\phi_{2,2}-\phi_{0,3}
\end{Bmatrix}
$$
is a basis for the subspace of $M_{0,\rho_{L_{4}}}^{!}$ consisting of elements whose $\phi_{j,k}$-component functions are all constants. In particular, the dimension of such subspaces is $\sigma(N)=\sum_{d|N}1$.
\end{Lemma}

\begin{proof}
Setting up the equations
\begin{align*}
\rho_{L_{N}}(T)\left(\sum_{k=0}^{N-1}\sum_{j=0}^{N-1}a_{j,k}\phi_{j,k}\right)
&=\sum_{k=0}^{N-1}\sum_{j=0}^{N-1}a_{j,k}\phi_{j,k},\\
\rho_{L_{N}}(S)\left(\sum_{k=0}^{N-1}\sum_{j=0}^{N-1}a_{j,k}\phi_{j,k}\right)
&=\sum_{k=0}^{N-1}\sum_{j=0}^{N-1}a_{j,k}\phi_{j,k},
\end{align*}
expanding the left hand side by the Weil representation, and equating the coefficients by the linear independence of $\phi_{j,k}$, we can obtain the desired results after some routine calculations.
\end{proof}

\begin{Lemma}
\label{vecfn}
Let $c_{s,N}^{-1}$ be the inverse of $c_{s}$ in $(\mathbb{Z}/N\mathbb{Z})^{\times}$ when $m_{s}=1$, and $c_{s,N}^{-1}=m_{s}$ otherwise. For $s\in\mathcal{C}(\Gamma_{1}(N))$, and integers $j$ and $k$, define $t_{s,j,k}$ by $t_{s,j,k}/h_{s}\equiv jk/N\pmod{1}$.
Let $\pi_{N}=\pi_{N}(\tau)$ be a Hauptmodul for a genus zero group $\Gamma_{0}(N)$ and $f_{N}=f_{N}(\tau)$ be the $\Gamma_{0}(N)$-induction of $\pi_{N}(\tau)$ against $\phi_{0,0}$ defined by
\begin{equation}
\label{deffn}
f_{N}=\sum_{M\in\Gamma_{0}(N)\backslash{\rm SL}_{2}(\mathbb{Z})}\left.\pi_{N}\right|M\cdot\rho_{L_{N}}(M^{-1})\phi_{0,0}.
\end{equation}
Then $f_{N}$ is in $ M^{!}_{0,\rho_{L_{N}}}$, and
$$
f_{N}=\frac{2}{|\Gamma_{0}(N):\Gamma_{1}(N)|}\sum_{M\in\Gamma_{1}(N)\backslash{\rm SL}_{2}(\mathbb{Z})}\frac{1}{2}\left.\pi_{N}\right|M\cdot\rho_{L_{N}}(M^{-1})\phi_{0,0}
$$ 
where 
\begin{align}
\label{fnpi}
&f_{N}=\frac{2}{\lambda_{2,N}|\Gamma_{0}(N):\Gamma_{1}(N)|}\Bigg[\sum_{\substack{s\in\mathcal{C}(\Gamma_{1}(N))\\\text{$s$ regular}}}\left(\left.\pi_{N}\right|M_{s}\right)_{0}\phi_{0,0}+\sum_{\substack{s\in\mathcal{C}(\Gamma_{1}(N))\\\text{$s$ regular}\\m_{s}\ne N}}\left(\left.\pi_{N}\right|M_{s}\right)_{0}\sum_{j=1}^{ h_{s}-1}\phi_{jm_{s},0}\\
&\qquad\quad\qquad\qquad\qquad\qquad+\sum_{\substack{s\in\mathcal{C}(\Gamma_{1}(N))\\\text{$s$ regular}\\m_{s}\ne N}}\left(\left.\pi_{N}\right|M_{s}\right)_{0}\sum_{k=1}^{ h_{s}-1}\phi_{0,km_{s}}\nonumber\\
&\qquad\quad\qquad\qquad\qquad\qquad+\sum_{\substack{s\in\mathcal{C}(\Gamma_{1}(N))\\\text{$s$ regular}\\m_{s}\ne N}}\sum_{k=1}^{ h_{s}-1}\sum_{j=1}^{ h_{s}-1}e\left(-\frac{d_{s}c_{s,N}^{-1}m_{s}^{2}jk}{N}\right)\left(\left.\pi_{N}\right|M_{s}\right)_{t_{s}}\phi_{jm_{s},km_{s}}\nonumber\\
&\quad\qquad\qquad\qquad\qquad\qquad+\sum_{\substack{s\in\mathcal{C}(\Gamma_{1}(N))\\\text{$s$ irregular}}}\frac{1}{h_{s}}\left(\left.\pi_{N}\right|M_{s}\right)\sum_{k=0}^{ h_{s}-1}\sum_{j=0}^{ h_{s}-1}e\left(-\frac{d_{s}m_{s}jk}{N}\right)\phi_{jm_{s},km_{s}}\Bigg],\nonumber
\end{align}
and $\lambda_{2,N}=2$ or $1$ depending on whether $N=2$ or not.
\end{Lemma}

\begin{proof}
These follow from \cite[Theorem 5.4]{bor6} and \cite[Theorem 3.7]{sche} by setting $D=L_{N}'/L_{N}$ with quadratic form $Q(\cdot)=\det(\cdot)$ and $e^{\gamma}=\phi_{0,0}$, and realizing that 
$$
D_{c_{s}}=\left\{\phi_{j h_{s},\,k h_{s}}\right\}_{0\leq j,\,k\leq m_{s}-1}\quad\mbox{and}\quad D^{c_{s}*}=\left\{\phi_{jm_{s},\,km_{s}}\right\}_{0\leq j,\,k\leq  h_{s}-1}.
$$
\end{proof}

\begin{Lemma}
\label{lemFN}
Let $\pi_{N}=\pi_{N}(\tau)$ be a Hauptmodul for genus zero group $\Gamma_{0}(N)$, and let
$$
\left.\pi_{N}\right|M_{s}=\sum_{n=-1}^{\infty}A_{s}(n)q^{n/\tilde{h}_{s}}
$$
where $\tilde{h}_{s}=1$ if $N=4$ and $s=1/2$, otherwise, $\tilde{h}_{s}=h_{s}$ as defined in Section \ref{intro}. Let $f_{N}=f_{N}(\tau)$ be the $\Gamma_{0}(N)$-induction as defined in Lemma $\ref{vecfn}$.
For $N\ne4$, let $F_{N}=F_{N}(\tau)$ be defined by
\begin{align}
\label{FN}
F_{N}=f_{N}-&\frac{2}{\lambda_{2,N}|\Gamma_{0}(N):\Gamma_{1}(N)|}\\
&\times\left[\sum_{\substack{s\in\mathcal{C}(\Gamma_{1}(N))}}A_{s}(0)\sum_{k=0}^{N-1}\phi_{0,k}+\sum_{\substack{s\in\mathcal{C}(\Gamma_{1}(N))\\m_{s}\ne N}}A_{s}(0)\left(\sum_{k=0}^{m_{s}-1}\sum_{j=0}^{ h_{s}-1}\phi_{jm_{s},\,k h_{s}}-\sum_{k=0}^{N-1}\phi_{0,\,k}\right)\right].\nonumber
\end{align}
For $N=4$, let $F_{4}=F_{4}(\tau)$ be defined by
\begin{align}
\label{F4}
F_{4}&=f_{4}-\left(A_{1/4}(0)+A_{0/1}(0)+\frac{1}{2}A_{1/2}(0)\right)\sum_{k=0}^{3}\phi_{0,k}-A_{0/1}(0)\sum_{k=1}^{3}(\phi_{k,0}-\phi_{0,k})\\
&\quad\qquad-\frac{1}{2}A_{1/2}(0)\left(\phi_{2,0}-\phi_{0,1}+\phi_{2,2}-\phi_{0,3}\right).\nonumber
\end{align}
Then we have
\begin{enumerate}
\item{$F_{N}(\tau)$ is in $M^{!}_{0,\rho_{L_{N}}}$, and is invariant under ${\rm Aut}(L_{N}'/ L_{N})$,}
\item{$c(0,\mu_{0,0})=0$ and the $\phi_{0,0}$--component of $F_{N}$ has Fourier expansion
$$
q^{-1}+\frac{2}{\lambda_{2,N}|\Gamma_{0}(N):\Gamma_{1}(N)|}\left[\sum_{\substack{s\in\mathcal{C}(\Gamma_{1}(N))\\\text{$s$ regular}}}\sum_{n=1}^{\infty}A_{s}(nh_{s})q^{n}+\sum_{\substack{s\in\mathcal{C}(\Gamma_{1}(N))\\\text{$s$ irregular}}}\frac{1}{h_{s}}\sum_{n=1}^{\infty}A_{s}(n)q^{n}\right],
$$
}

\item{$c(0,\mu_{j,0})=0$ for $0\leq j\leq N-1$,}
\item{for $d|N$ and $d\ne N$,
$$
\sum_{k=0}^{\frac{N}{d}-1}\sum_{j=0}^{d-1}c(0,\mu_{j\frac{N}{d},kd})=24.
$$}
\end{enumerate}
\end{Lemma}

\begin{Remark}
The different presentation of the case $N=4$ from the other cases mainly follows from the fact that $\Gamma_{0}(4)$ has an irregular cusp at~$\frac{1}{2}$.
\end{Remark}

\begin{proof}
Assertion (1) follows from Lemmas \ref{wt0} and \ref{vecfn}. 

By collecting the terms attached to $\phi_{0,0}$ in \eqref{fnpi}, we obtain that the $\phi_{0,0}$--component of $f_{N}$ is
\begin{align}
\label{phi00}
&\frac{2}{\lambda_{2,N}|\Gamma_{0}(N):\Gamma_{1}(N)|}\left[\sum_{\substack{s\in\mathcal{C}(\Gamma_{1}(N))\\\text{$s$ regular}}}\left(\left.\pi_{N}\right|M_{s}\right)_{0}+\sum_{\substack{s\in\mathcal{C}(\Gamma_{1}(N))\\\text{$s$ irregular}}}\frac{1}{h_{s}}\left(\left.\pi_{N}\right|M_{s}\right)\right]\nonumber\\
&=q^{-1}+\frac{2}{\lambda_{2,N}|\Gamma_{0}(N):\Gamma_{1}(N)|}\left[\sum_{\substack{s\in\mathcal{C}(\Gamma_{1}(N))\\\text{$s$ regular}}}\sum_{n=0}^{\infty}A_{s}(nh_{s})q^{n}+\sum_{\substack{s\in\mathcal{C}(\Gamma_{1}(N))\\\text{$s$ irregular}}}\frac{1}{h_{s}}\sum_{n=0}^{\infty}A_{s}(n)q^{n}\right]\nonumber\\
&=q^{-1}+\frac{2}{\lambda_{2,N}|\Gamma_{0}(N):\Gamma_{1}(N)|}\Bigg[\sum_{\substack{s\in\mathcal{C}(\Gamma_{1}(N))\\\text{$s$ regular}}}A_{s}(0)+\sum_{\substack{s\in\mathcal{C}(\Gamma_{1}(N))\\\text{$s$ irregular}}}\frac{1}{h_{s}}A_{s}(0)\\
&\qquad\qquad\qquad\qquad\qquad\qquad\qquad+\sum_{\substack{s\in\mathcal{C}(\Gamma_{1}(N))\\\text{$s$ regular}}}\sum_{n=1}^{\infty}A_{s}(nh_{s})q^{n}+\sum_{\substack{s\in\mathcal{C}(\Gamma_{1}(N))\\\text{$s$ irregular}}}\frac{1}{h_{s}}\sum_{n=1}^{\infty}A_{s}(n)q^{n}\Bigg].\nonumber
\end{align}
Then Assertion (2) follows from \eqref{phi00} and the definition of $F_{N}$.

By extracting the constant terms attached to $\phi_{j,0}$ for $1\leq j\leq N-1$ in \eqref{fnpi} of $F_{N}$, we obtain
\begin{equation}
\label{ext}
\frac{2}{\lambda_{2,N}|\Gamma_{0}(N):\Gamma_{1}(N)|}\left[\sum_{\substack{s\in\mathcal{C}(\Gamma_{1}(N))\\\text{$s$ regular}\\m_{s}\ne N}}A_{s}(0)\sum_{j=1}^{ h_{s}-1}\phi_{jm_{s},0}+\sum_{\substack{s\in\mathcal{C}(\Gamma_{1}(N))\\\text{$s$ irregular}}}\frac{1}{h_{s}}A_{s}(0)\sum_{j=1}^{ h_{s}-1}\phi_{jm_{s},0}\right].
\end{equation}
Then Assertion (3) follows from \eqref{ext} and the definition of $F_{N}$.

For Assertion (4), by \eqref{fnpi}, \eqref{FN} and \eqref{F4}, it can be verified case by case that
\begin{equation}
\label{sumc0k}
\sum_{k=0}^{\frac{N}{d}-1}\sum_{j=0}^{d-1}c(0,\phi_{j\frac{N}{d},kd})=-\frac{2}{\lambda_{2,N}|\Gamma_{0}(N):\Gamma_{1}(N)|}\sum_{s\in \mathcal{C}(\Gamma_{1}(N))}A_{s}(0)\left(\frac{d(m_{s},N/d)^{2}-N}{m_{s}}\right).
\end{equation}
Let 
$$
H_{d}(\tau)={\frac{N}{d}E_{2}\left(\frac{N}{d}\tau\right)-E_{2}(\tau)}
$$
where $E_{2}(\tau)$ is the normalized weight~2 Eisenstein series. It is well known \cite[Section 1.2]{diamond} that $H_{d}(\tau)$ is a weight~2 modular form for $\Gamma_{0}(N/d)$, so it is a weight ~2 modular form for $\Gamma_{1}(N)$. Let $B_{s}(0)$ be the constant term of the Fourier expansion of $H_{d}(\tau)$ at the cusp $s=a_{s}/c_{s}\in\mathcal{C}(\Gamma_{1}(N))$. It is easy to show that
$$
B_{s}(0)=\frac{(m_{s},N/d)^{2}}{N/d}-1
$$
and
$$
H_{d}(\tau)=\frac{N}{d}-1+{24}q+O(q^{2}).
$$
Since $\pi_{N}(\tau)$ is a weight~0 weakly holomorphic modular form for $\Gamma_{0}(N)$, then it is also on $\Gamma_{1}(N)$, and by Serre duality \cite[Theorem 3.1]{bor2}, we have
\begin{align}
\label{serdu}
0&=24\times\frac{\lambda_{2,N}|\Gamma_{0}(N):\Gamma_{1}(N)|}{2}+\sum_{\substack{s\in\mathcal{C}(\Gamma_{1}(N))}}h_{s}B_{s}(0)A_{s}(0)\nonumber\\
&=24\times\frac{\lambda_{2,N}|\Gamma_{0}(N):\Gamma_{1}(N)|}{2}+\sum_{\substack{s\in\mathcal{C}(\Gamma_{1}(N))}} h_{s}\left(\frac{(m_{s},N/d)^{2}}{N/d}-1\right)A_{s}(0)\nonumber\\
&=24\times\frac{\lambda_{2,N}|\Gamma_{0}(N):\Gamma_{1}(N)|}{2}+\sum_{s\in \mathcal{C}(\Gamma_{1}(N))}A_{s}(0)\left(\frac{d(m_{s},N/d)^{2}-N}{m_{s}}\right).
\end{align}
Finally, Assertion (4) follows from \eqref{sumc0k} and \eqref{serdu}.
\end{proof}

\section{Proofs of Theorem \ref{borthm} and Corollaries \ref{diffpiexp} and \ref{PN}}
\subsection{Proof of Theorem \ref{borthm}}
This subsection is devoted to the proof of our first main result, Theorem \ref{borthm}. We rely heavily on Borcherds Theorem, especially the third part of it. To assist in understanding the proof, we first review the definitions of the notation used in Theorem \ref{bors}(3) for our case. Let $M_{\ell}=L\cap(\mathbb{Q}\ell+\mathbb{Q}\ell')^{\perp}$ be the Lorentzian lattice of $L$ associated to $\ell$ and $\ell'$. Assume that $(\ell,L)=N_{\ell}\Z$. Choose a $\xi\in L$ such that $(\xi,\ell)=N_{\ell}$. Let $L_{\ell}'$ be a sublattice of $L'$ defined by
$$
L_{\ell}'=\{x\in L'|\,\text{$(\ell,x)\equiv0\pmod{N_{\ell}}$}\}.
$$
Then there is a projection
$$
p:L_{\ell}'\to M_{\ell}',\qquad p(x)=x_{M}-\frac{(x,\ell)}{N_{\ell}}\xi_{M},
$$
where $x_{M}$ and $\xi_{M}$ are the orthogonal projections of $x,\xi\in V$ to $M_{\ell,\Q}$. So it induces a projection, which is also denoted by $p$, from $L_{\ell}'/L$ to $M_{\ell}'/M_{\ell}$. Next, we define the Weyl chamber $W_{\ell_{M}}$ for 
$$
f=\sum_{\mu\in L'/L}f_{\mu}\phi_{\mu}\in M_{0,\omega_{L}}^{!}.
$$
Let $\Gr(M_{\ell})$ be the Grassmannian of negative lines of $M_{\ell,\RR}$, which is a real manifold of dimension~1. For $\lambda\in M_{\ell}'/M_{\ell}$ and $n\in\Q$, let 
$$
Z_{M}(n,\lambda)=\{z\in\Gr(M_{\ell})|\,\text{$z\perp x$ for some $x\in\lambda+M_{\ell}$ with $Q(x)=n$}\},
$$
which is either empty or a real divisor of $\Gr(M_{\ell})$. The Weyl chambers $W_{f}$ associated with $f$ are the connected components of 
$$
\Gr(M_{\ell})-\bigcup_{\mu\in L_{\ell}'/L}\bigcup_{\substack{n>0\\c(-n,\mu)\ne0}}Z_{M}(n,p(\mu)).
$$

Let $\ell_{M}\in M_{\ell}$ and $\ell_{M}'\in M_{\ell}'$ be isotropic elements with $(\ell_{M},\ell_{M}')=1$, whose existence is guaranteed in our case. For general case, we refer the reader to \cite[Section 9]{bor} and \cite[Section 5]{jan2}. Choose a Weyl chamber such that $\ell_{M}$ is contained in its closure, and denote such Weyl chamber by $W_{f,\ell_{M}}$. We now define its Weyl vector $\rho(W_{f,\ell_{M}},f)$. Define
$$
f_{M}=\sum_{\lambda\in M_{\ell}'/M_{\ell}}f_{M,\lambda}\phi_{M,\lambda}
$$
where $\phi_{M,\lambda}\in\C[M_{\ell}'/M_{\ell}]$ and 
$$
f_{M,\lambda}=\sum_{\substack{\mu\in L_{\ell}'/L\\ p(\mu)=\lambda+M}}f_{\mu}.
$$ 
Let $P_{\ell_{M}}$ be the Lorentzian sublattice of $M_{\ell}$ associated with $\ell_{M}$. Clearly, $P_{\ell_{M}}=\{0\}$ and $P_{\ell_{M}}'/P_{\ell_{M}}$ is trivial. Construct $f_{P}$ from $f_{M}$ in the same way as $f_{M}$ constructed from $f$, and obtain
$$
f_{P}=\sum_{\lambda\in M_{\ell,\ell_{M}}'/M_{\ell}}\sum_{\substack{\mu\in L_{\ell}'/L\\ p(\mu)=\lambda+M_{\ell}}}f_{\mu}.
$$
Then the Weyl vector $\rho(W_{f,\ell_{M}},f)$ is defined by
$$
\rho(W_{f,\ell_{M}},f)=\rho_{\ell_{M}}\ell_{M}+\rho_{\ell_{M}'}\ell_{M}'
$$
where
\begin{align*}
\rho_{\ell_{M}'}&=-1+\frac{c_{P}(0)}{24}=-1+\frac{1}{24}\sum_{\lambda\in M_{\ell,\ell_{M}}'/M_{\ell}}\sum_{\substack{\mu\in L_{\ell}'/L\\ p(\mu)=\lambda+M_{\ell}}}c(0,\mu),\\
\rho_{\ell_{M}}&=-\frac{1}{4}\sum_{\lambda\in M_{\ell,\ell_{M}}'/M_{\ell}}c_{M}(0,\lambda)B_{2}((\lambda,\ell_{M}')),
\end{align*}
$B_{2}(x)=\{x\}^{2}-\{x\}+\frac{1}{6}$ is the second Bernoulli polynomial, and $\{x\}:=x-[x]$.

\begin{proof}[Proof of Theorem \ref{borthm}]
Let $\ell_{s}=\begin{pmatrix}0&a_{s}\\0&c_{s}\end{pmatrix}\in L_{N}$ and $\ell_{s}'=\begin{pmatrix}-b_{s}&0\\-d_{s}&0\end{pmatrix}\in L_{N}'$. Then we have
$$
L_{N,\ell_{s}}'=\begin{pmatrix}\Z&\Z/N\\m_{s}\Z&\Z\end{pmatrix},\qquad L_{N,\ell_{s}}'/L_{N}=\left\{\mu_{j,m_{s}k}=\begin{pmatrix}0&j/N\\m_{s}k&0\end{pmatrix}\right\}_{\substack{0\leq j\leq N-1\\0\leq k\leq h_{s}-1}}
$$
$$
M_{s}:=M_{\ell_{s}}=\left\{\left.\begin{pmatrix} a_{s}h_{s}x&b_{s}y\\ c_{s}h_{s}x&d_{s}y\end{pmatrix}\right|\,x,y\in\Z\right\}\quad\mbox{and}\quad M_{s}':=M_{\ell_{s}}'=\left\{\left.\begin{pmatrix}a_{s}x& \frac{b_{s}}{h_{s}}y\\c_{s}x& \frac{d_{s}}{h_{s}}y\end{pmatrix}\right|\,x,y\in\Z\right\},
$$
and
$$
z_{\ell_{s}}=\begin{pmatrix}-a_{s}z_{1}&b_{s}z_{2}\\-c_{s}z_{1}&d_{s}z_{2}\end{pmatrix}\quad\mbox{and}\quad w(z_{\ell_{s}})=\begin{pmatrix}-a_{s}z_{1}-b_{s}&b_{s}z_{2}+a_{s}z_{1}z_{2}\\-c_{s}z_{1}-d_{s}&d_{s}z_{2}+c_{s}z_{1}z_{2}\end{pmatrix}.
$$
Also, we deduce that
$$
\Gr(M_{s})=\left\{\mathbb{R}\left.\begin{pmatrix}-a_{s}x&b_{s}\\-c_{s}x&d_{s}\end{pmatrix}\right|\,x>0\right\}
$$
and
\begin{align*}
Z_{M_{s}}(1,\mu_{0,0})&=\left\{Z\in\Gr(M_{s})|\,\text{$Z\perp X$ for some $X\in \mu_{0,0}+M_{s}$ with $Q(X)=1$}\right\}\\
&=\begin{cases}\varnothing&\mbox{if $m_{s}\ne N$,}\\\left\{\mathbb{R}\begin{pmatrix}-a_{s}&b_{s}\\-c_{s}&d_{s}\end{pmatrix}\right\}&\mbox{if $m_{s}=N$.}\end{cases}
\end{align*}
And thus
$$
\Gr(M_{s})-Z_{M_{s}}(1,\mu_{0,0})=\begin{cases}\Gr(M_{s})&\mbox{if $m_{s}\ne N$,}\\\left\{\mathbb{R}\left.\begin{pmatrix}-a_{s}x&b_{s}\\-c_{s}x&d_{s}\end{pmatrix}\right|\,0<x<1\right\}\cup\left\{\mathbb{R}\left.\begin{pmatrix}-a_{s}x&b_{s}\\-c_{s}x&d_{s}\end{pmatrix}\right|\,x>1\right\}&\mbox{if $m_{s}=N$.}\end{cases}
$$
Let $\ell_{M_{s}}=\begin{pmatrix} -a_{s}h_{s}&0\\ -c_{s}h_{s}&0\end{pmatrix}$ and $\ell_{M_{s}}'=\begin{pmatrix}0& -\frac{b_{s}}{h_{s}}\\0& -\frac{d_{s}}{h_{s}}\end{pmatrix}$. Then the Weyl chamber whose closure contains $\Q\ell_{M_{s}}$ is
$$
W_{\ell_{M_{s}}}=\begin{cases}\left\{\mathbb{R}\left.\begin{pmatrix}-a_{s}x&b_{s}\\-c_{s}x&d_{s}\end{pmatrix}\right|\,x>0\right\}&\mbox{if $m_{s}\ne N$,}\\\\\left\{\mathbb{R}\left.\begin{pmatrix}-a_{s}x&b_{s}\\-c_{s}x&d_{s}\end{pmatrix}\right|\,x>1\right\}&\mbox{if $m_{s}=N$,}\end{cases}
$$
and thus for $X=\begin{pmatrix}a_{s}n&- \frac{b_{s}}{h_{s}}m\\c_{s}n&- \frac{d_{s}}{h_{s}}m\end{pmatrix}\in M_{s}'$, we have that
\begin{equation}
\label{XW}
(X,W_{\ell_{M_{s}}})>0\quad\mbox{if and only if}\quad\begin{cases}\text{$m,\,n\geq0$ and $m^{2}+n^{2}>0$}&\mbox{if $m_{s}\ne N$,}\\
\text{$m\geq0$, $m+n\geq0$ and $m^{2}+n^{2}>0$}&\mbox{if $m_{s}= N$,}\end{cases}
\end{equation}
and we can check that $(X,z_{\ell_{s}})=\frac{m_{s}}{N}z_{1}m+z_{2}n=\frac{1}{h_{s}}z_{1}m+z_{2}n$.
We also have
$$
M_{s,\ell_{M_{s}}}'=\left\{\left.\begin{pmatrix}a_{s}x&b_{s}y\\c_{s}x&d_{s}y\end{pmatrix}\right|\,x,y\in\Z\right\}\quad\mbox{and}\quad M_{s,\ell_{M_{s}}}'/M_{s}=\left\{\begin{pmatrix}a_{s}k&0\\c_{s}k&0\end{pmatrix}\right\}_{0\leq k\leq h_{s}-1}.
$$
Let $\tilde{x},\tilde{y}\in\Z$ such that $c_{s}\tilde{x}-aN\tilde{y}=m_{s}$. Then $p:L_{N,\ell_{s}}'/L_{N}\to M_{s}'/M_{s}$ is
$$
p\left(\begin{pmatrix}0&j/N\\m_{s}k&0\end{pmatrix}\right)=\begin{pmatrix}a_{s}k\tilde{x}&-\frac{b_{s}c_{s}}{N}j\\c_{s}k\tilde{x}&-\frac{d_{s}c_{s}}{N}j\end{pmatrix},
$$
and thus
$$
F_{N,P}=\sum_{\lambda\in M_{s,\ell_{M_{s}}}'/M}\left(\sum_{\substack{\mu\in L_{N,\ell_{s}}'/L_{N}\\p(\mu)=\lambda}}F_{N,\phi_{\mu}}\right)=\sum_{k=0}^{ h_{s}-1}\sum_{j=0}^{m_{s}-1}F_{N,\phi_{j h_{s},km_{s}}}.
$$
Thus, 
$$
c_{P}(0,\phi_{0,0})=\sum_{k=0}^{ h_{s}-1}\sum_{j=0}^{m_{s}-1}c(0,\phi_{j h_{s},km_{s}})=\begin{cases}24&\mbox{if $m_{s}\ne N$,}\\0&\mbox{if $m_{s}=N$,}\end{cases}
$$
and
$$
\rho_{\ell_{M_{s}}'}=\begin{cases}0&\mbox{if $m_{s}\ne N$,}\\-1&\mbox{if $m_{s}=N$.}\end{cases}
$$
Therefore, the Weyl vector of $W_{\ell_{M_{s}}}$ associated to $F_{N}$ is
$$
\rho(W_{\ell_{M_{s}}},F_{N})=\begin{cases}\rho_{\ell_{M_{s}}}\ell_{M_{s}}&\mbox{if $m_{s}\ne N$,}\\-\ell_{M_{s}}'&\mbox{if $m_{s}=N$,}\end{cases}
$$
and thus
\begin{equation}
\label{weylvec}
(z_{\ell_{s}},\rho(W_{\ell_{M_{s}}},F_{N}))=\begin{cases}-h_{s}\rho_{\ell_{M_{s}}}z_{2}&\mbox{if $m_{s}\ne N$,}\\-z_{1}&\mbox{if $m_{s}=N$.}\end{cases}
\end{equation}
We can also show that $h_{s}\rho_{\ell_{M_{s}}}=1$, but we do not need this fact in our proof, so we leave the details of computations to the reader. Similarly, if we let $\ell_{s}=\begin{pmatrix}c_{s}&-a_{s}\\0&0\end{pmatrix}\in L_{N}$ and $\ell_{s}'=\begin{pmatrix}0&0\\d_{s}&-b_{s}\end{pmatrix}\in L_{N}'$, then we have 
$$
M_{s}=\left\{\left.\begin{pmatrix}-d_{s}y&b_{s}y\\ c_{s}h_{s}x&-a_{s}h_{s}x\end{pmatrix}\right|\,x,y\in\Z\right\}\quad\mbox{and}\quad M_{s}'=\left\{\left.\begin{pmatrix}-\frac{d_{s}}{h_{s}}y& \frac{b_{s}}{h_{s}}y\\c_{s}x&-a_{s}x\end{pmatrix}\right|\,x,y\in\Z\right\},
$$
$$
\ell_{M_{s}}=\begin{pmatrix}0&0\\ c_{s}h_{s}&- a_{s}h_{s}\end{pmatrix}\quad\mbox{and}\quad\ell_{M_{s}}'=\begin{pmatrix}- \frac{d_{s}}{h_{s}}& \frac{b_{s}}{h_{s}}\\0&0\end{pmatrix},
$$
$$
z_{\ell_{s}}=\begin{pmatrix}d_{s}z_{1}&-b_{s}z_{1}\\c_{s}z_{2}&-a_{s}z_{2}\end{pmatrix}\quad\mbox{and}\quad w(z_{\ell_{s}})=\begin{pmatrix}d_{s}z_{1}+c_{s}z_{1}z_{2}&-b_{s}z_{1}-a_{s}z_{1}z_{2}\\c_{s}z_{2}+d_{s}&-a_{s}z_{2}-b_{s}\end{pmatrix},
$$
$$
W_{\ell_{M_{s}}}=\begin{cases}\left\{\mathbb{R}\left.\begin{pmatrix}d_{s}&-b_{s}\\c_{s}x&-a_{s}x\end{pmatrix}\right|\,x>0\right\}&\mbox{if $m_{s}\ne N$,}\\\\\left\{\mathbb{R}\left.\begin{pmatrix}d_{s}&-b_{s}\\c_{s}x&-a_{s}x\end{pmatrix}\right|\,x>1\right\}&\mbox{if $m_{s}=N$,}\end{cases}
$$
\begin{equation}
\label{XW2}
(X,W_{\ell_{M_{s}}})>0\quad\mbox{if and only if}\quad\begin{cases}\text{$m,\,n\geq0$ and $m^{2}+n^{2}>0$}&\mbox{if $m_{s}\ne N$,}\\
\text{$n\geq0$, $m+n\geq0$ and $m^{2}+n^{2}>0$}&\mbox{if $m_{s}= N$,}\end{cases}
\end{equation}
for $X=\begin{pmatrix}- \frac{d_{s}}{h_{s}}n& \frac{b_{s}}{h_{s}}n\\-c_{s}m&a_{s}m\end{pmatrix}\in M_{s}'$, and
$$
\rho(W_{\ell_{M_{s}}},F_{N})=\begin{cases}\rho_{\ell_{M_{s}}}\ell_{M_{s}}&\mbox{if $m_{s}\ne N$,}\\-\ell_{M_{s}}'&\mbox{if $m_{s}=N$.}\end{cases}
$$
Also, we can check $(X,z_{\ell_{s}})=z_{1}m+\frac{1}{h_{s}}z_{2}n$, and
\begin{equation}
\label{vec2}
(z_{\ell_{s}},\rho(W_{\ell_{M_{s}}},F_{N}))=\begin{cases}-h_{s}\rho_{\ell_{M_{s}}}z_{1}&\mbox{if $m_{s}\ne N$,}\\-z_{2}&\mbox{if $m_{s}=N$.}\end{cases}
\end{equation}

Now we are ready for the proof of Theorem \ref{borthm}. We aim to show that
$$
\pi_{N}(z_{1})-\pi_{N}(z_{2})=\Psi(z,F_{N}).
$$
We first note by Borcherds Theorem, Proposition \ref{yyprop} and Lemma \ref{lemFN}(ii) that $\Psi(z,F_{N})$ can be viewed as either a meromorphic function on the Shimura variety $\mathcal{X}_{K_{N}}$ or a meromorphic function on $Y_{0}(N)\times Y_{0}(N)$, and $h_{N}(z_{1},z_{2}):=\pi_{N}(z_{1})-\pi_{N}(z_{2})$ is a meromorphic function on  $Y_{0}(N)\times Y_{0}(N)$.
Let 
$$
g(z_{1},z_{2})=\frac{\Psi(z,F_{N})}{h_{N}(z_{1},z_{2})}.
$$
Then $g(z_{1},z_{2})$ is a meromorphic function on $Y_{0}(N)\times Y_{0}(N)$ with no zeros or poles by Lemma \ref{zerolem}.
Let us fix $z_{2}\in\HH$. Then $g_{1}(z_{1}):=g(z_{1},z_{2})$ is a meromorphic function in $z_{1}$ on $Y_{0}(N)$. Let us investigate the behavior of $g_{1}(z_{1})$ at the cusps of $Y_{0}(N)$. By Proposition \ref{yyprop}, we can know that the Fourier expansion of $\Psi(z,F_{N})$ at the cusp $(s=a_{s}/c_{s},i\infty)$ is equal to the Borcherds product expansion of $\Psi(z,F_{N})$ at the cusp $\Q\ell_{s}$ where $\ell_{s}=\begin{pmatrix}0&a_{s}\\0&c_{s}\end{pmatrix}$. Then by Theorem \ref{bors}(3) together with \eqref{XW} and \eqref{weylvec}, we have
$$
\Psi(z,F_{N})=\begin{cases}\displaystyle{C_{s}q_{2}^{-h_{s}\rho_{\ell_{M_{s}}}}\prod_{\substack{m,n\geq0\\m^{2}+n^{2}>0}}\prod_{\substack{\mu\in L_{N,\ell_{s}}'/L_{N}\\ p(\mu)=\begin{pmatrix}a_{s}n& -\frac{b_{s}}{h_{s}}m\\c_{s}n& -\frac{d_{s}}{h_{s}}m\end{pmatrix}}}\left(1-q_{1}^{m/h_{s}}q_{2}^{n}e((\mu,\ell_{s}'))\right)^{c(mn,\phi_{\mu})}}&\mbox{if $m_{s}\ne N$,}\\
\displaystyle C_{s}(q_{1}^{-1}-q_{2}^{-1})\prod_{\substack{m\geq0\\m+n\geq0\\m^{2}+n^{2}>0}}\prod_{\substack{\mu\in L_{N,\ell_{s}}'/L_{N}\\ p(\mu)=\begin{pmatrix}a_{s}n& -{b_{s}}m\\c_{s}n& -{d_{s}}m\end{pmatrix}}}\left(1-q_{1}^{m}q_{2}^{n}e((\mu,\ell_{s}'))\right)^{c(mn,\phi_{\mu})}&\mbox{if $m_{s}=N$,}\end{cases}
$$
near the cusp $\Q\ell_{s}$ for some nonzero constant $C_{s}$ depending on the choice of cusp $s$, which can be identified by the condition given in Borcherds' Theorem. Then when $z_{2}\in\HH$ is fixed, as a meromorphic function on $Y_{0}(N)$, the order of $\Psi(z,F_{N})$ at the cusp $s=a_{s}/c_{s}$ can be computed 
$$
\mbox{ord}_{s}(\Psi(z,F_{N}))=\begin{cases}0&\mbox{if $m_{s}\ne N$,}\\-1&\mbox{if $m_{s}=N$.}\end{cases}
$$
Also, we can easily compute the Fourier expansion of $h_{N}(z_{1},z_{2})$ at the same cusp, and obtain the order of $h_{N}(z_{1},z_{2})$ at the cusp $s=a_{s}/c_{s}$ when $z_{2}\in\HH$ is fixed, which is
$$
\mbox{ord}_{s}(h_{N})=\begin{cases}0&\mbox{if $m_{s}\ne N$,}\\-1&\mbox{if $m_{s}=N$.}\end{cases}
$$
Therefore, as a meromorphic function on $Y_{0}(N)$, the modular function $g_{1}(z_{1})$ is holomorphic at all of the cusps, and thus $g(z_{1},z_{2})$ is constant on $Y_{0}(N)\times\{z_{2}\}$. Similarly, if we fix $z_{1}\in\HH$, by Theorem~\ref{bors}(3), \eqref{XW2} and \eqref{vec2}, we can show that $g(z_{1},z_{2})$ is constant on $\{z_{1}\}\times Y_{0}(N)$. Hence, $g(z_{1},z_{2})$ is constant on $Y_{0}(N)\times Y_{0}(N)$, which is~1 by comparing the Fourier expansions of $\Psi(z,F_{N})$ and $h_{N}(z_{1},z_{2})$ at $(i\infty,i\infty)=(1/N,1/N)$, and this completes the proof.
\end{proof}

\subsection{Proofs of Corollaries \ref{diffpiexp} and \ref{PN}}We end this section with the proofs to Corollaries \ref{diffpiexp} and \ref{PN}.
\begin{proof}[Proof of Corollary \ref{diffpiexp}]
We first note \eqref{fnpi}--\eqref{F4} that the $\phi_{j,0}$-component for $0\leq j\leq N-1$ of $F_{N}$ is 
\begin{align}
\label{cA}
&\sum_{\ell=-1}^{\infty}c(\ell,\mu_{j,0})q^{\ell}\\
&=\frac{2}{\lambda_{2,N}|\Gamma_{0}(N):\Gamma_{1}(N)|}\left\{\sum_{\substack{s\in\mathcal{C}(\Gamma_{1}(N))\\\text{$s$ regular}\\m_{s}|j}}\left[\left(\left.\pi_{N}\right|M_{s}\right)_{0}-A_{s}(0)\right]+\sum_{\substack{s\in\mathcal{C}(\Gamma_{1}(N))\\\text{$s$ irregular}\\m_{s}|j}}\frac{1}{h_{s}}\left[\left(\left.\pi_{N}\right|M_{s}\right)-A_{s}(0)\right]\right\}\nonumber\\
&=\sum_{\substack{d|N\\d|j}}\frac{2}{\lambda_{2,N}|\Gamma_{0}(N):\Gamma_{1}(N)|}\left\{\sum_{\substack{s\in\mathcal{C}(\Gamma_{1}(N))\\\text{$s$ regular}\\m_{s}=d}}\left[\left(\left.\pi_{N}\right|M_{s}\right)_{0}-A_{s}(0)\right]+\sum_{\substack{s\in\mathcal{C}(\Gamma_{1}(N))\\\text{$s$ irregular}\\m_{s}=d}}\frac{1}{h_{s}}\left[\left(\left.\pi_{N}\right|M_{s}\right)-A_{s}(0)\right]\right\}\nonumber\\
&=\sum_{\substack{d|N\\d|j}}\sum_{\ell=-1}^{\infty}A(\ell,d)q^{\ell}\nonumber\\
&=\sum_{\ell=-1}^{\infty}\left(\sum_{\substack{d|N\\d|j}}A(\ell,d)\right)q^{\ell}.
\end{align}
Now from the proof of Theorem \ref{borthm} and \eqref{cA} together with Borcherds Theorem, we deduce that
\begin{align*}
\pi_{N}(z_{1})-\pi_{N}(z_{2})
&=(q_{1}^{-1}-q_{2}^{-1})\prod_{m,n>0}\prod_{j=0}^{N-1}\prod_{\substack{d|N\\d|j}}\left(1-q_{1}^{m}q_{2}^{n}e\left(-j/N\right)\right)^{A(mn,d)}\\
&=(q_{1}^{-1}-q_{2}^{-1})\prod_{m,n>0}\prod_{{d|N}}\prod_{j'=0}^{\frac{N}{d}-1}\left(1-q_{1}^{m}q_{2}^{n}e\left(-j'd/N\right)\right)^{A(mn,d)}\\
&=(q_{1}^{-1}-q_{2}^{-1})\prod_{m,n>0}\prod_{{d|N}}\left(1-\left(q_{1}^{m}q_{2}^{n}\right)^{\frac{N}{d}}\right)^{A(mn,d)}.
\end{align*}
\end{proof}

\begin{proof}[Proof of Corollary \ref{PN}]
Taking the logarithmic derivative of both sides of \eqref{prodana}, we have that
\begin{align*}
&-\frac{1}{2\pi i}\frac{\pi_{N}'(z_{1})}{\pi_{N}(z_{1})-\pi_{N}(z_{2})}\\
&=\frac{1}{1-q_{2}^{-1}q_{1}}+\sum_{m,n>0}\sum_{d|N}\frac{A(mn,d)\left(q_{1}^{m}q_{2}^{n}\right)^{\frac{N}{d}}}{1-\left(q_{1}^{m}q_{2}^{n}\right)^{\frac{N}{d}}}\\
&=\sum_{n=0}^{\infty}(q_{2}^{-1}q_{1})^{n}+\sum_{m,n>0}\sum_{d|N}\sum_{\ell=1}^{\infty}A(mn,d)\left(q_{1}^{m}q_{2}^{n}\right)^{\frac{N}{d}\ell},
\end{align*}
and this proves the corollary.
\end{proof}


\part{Gross--Zagier Type CM Value Formulas}
\section{Big CM Cycles and Big CM Value Formula}
In this section, we briefly review the concepts of big CM cycles and big CM value formula (see\cite[Sec. 2--4]{bky} for details), based on which we realize a big CM cycle in the degenerate Hilbert modular surface $Y_{0}(N)\times Y_{0}(N)$ and prove Theorem~\ref{gztype} at the end of Subsection~\ref{pfgz}.

\subsection{Big CM Cycles in $\mathcal{X}_{K}$}
Let $F$ be a totally real number field of degree $d+1$ and $W$ be an $F$-quadratic space with an $F$-quadratic form $Q_{F}(\cdot)$ of signature $((2,0),\ldots,(2,0),(0,2))$ with respect to the $d+1$ embeddings $\{\sigma_{i}\}_{i=1}^{d+1}$ of $F$ such that ${\rm Res}_{F/\mathbb{Q}}W$ is a rational quadratic space of signature $(2d,2)$ with the quadratic form $Q(\cdot)={\rm tr}_{F/\mathbb{Q}}\circ Q_{F}(\cdot)$ induced from $Q_{F}(\cdot)$. Then we have an orthogonal direct sum decomposition
$$
{\rm Res}_{F/\mathbb{Q}}W=\bigoplus_{i=1}^{d+1}W_{\sigma_{i}}
$$
where $W_{\sigma_{i}}=W\otimes_{F,\sigma_{i}}\mathbb{R}$. The negative 2-plane $W_{\sigma_{d+1}}$ gives rise to a two-point (big CM points) subset $\{z_{\sigma_{d+1}}^{\pm}\}$ of $\mathbb{D}$. Let $T$ be the preimage of ${\rm Res}_{F/\mathbb{Q}}{\rm SO}(W)\subset {\rm SO}({\rm Res}_{F/\mathbb{Q}}W)$ in $H$, the general spin group of ${\rm Res}_{F/\mathbb{Q}}W$. Then we have the following commutative diagram.
\begin{center}
$\begin{CD}
1 @>>> \mathbb{G}_{m} @>>> T @>>> {\rm Res}_{F/\mathbb{Q}}{\rm SO}(W) @>>> 1 \\
 @.  @VVV @VVV @VVV @. \\ 
 1 @>>> \mathbb{G}_{m} @>>> H @>>> {\rm SO}(V) @>>> 1
\end{CD}$
\end{center}
and this implies that $T$ is a maximal torus associated to the CM number field $E=F(\sqrt{-\det W})$. And we obtain a so called big CM cycle in $\mathcal{X}_{K}$, the Shimura variety associated to the compact open subgroup $K$,
$$
Z(W,z_{\sigma_{d+1}}^{\pm})=T(\mathbb{Q})\backslash\left(\{z_{\sigma_{d+1}}^{\pm}\}\times T(\mathbb{A}_{\Q,f})/K_{T}\right),
$$
where $K_{T}=K\cap T(\mathbb{A}_{\Q,f})$. The CM cycle $Z(W,z_{\sigma_{d+1}}^{\pm})$ is defined over $F$, and the formal sum $Z(W)$ of all of its Galois conjugates is a 0-cycle in $\mathcal{X}_{K}$ defined over $\Q$.


\subsection{Big CM Value Formula} 
\label{bigcmfor}
Associated to the $F$-quadratic space $W$ and the additive adelic character $\psi_{F}=\psi\circ {\rm tr}_{F/
\mathbb{Q}}$ is a Weil representation $\omega=\omega_{\psi_{F}}$ of ${\rm SL}_{2}(\mathbb{A}_{\Q,f})$ (and thus $T(\mathbb{A_{\Q}})$) on $\mathcal{S}(W_{\mathbb{A}_{F}})=\mathcal{S}(V_{\mathbb{A}_{\mathbb{Q}}})$. Let $\chi=\chi_{E/F}$ be the quadratic Hecke character of $F$ associated to $E/F$. Then $\chi=\chi_{W}$ is also the quadratic Hecke character of $F$ associated to $W$, and there is an ${\rm SL}_{2}(\mathbb{A}_{F})$-equivariant map
$$
\lambda=\prod_{\nu}\lambda_{\nu}:\mathcal{S}(W_{\mathbb{A}_{F}})\to I(0,\chi)
$$
via $\lambda(\phi)(g)=\omega(g)\phi(0)$. where $I(s,\chi)={\rm Ind}^{{\rm SL}_{2}(\mathbb{A}_{F})}_{B_{\mathbb{A}_{F}}}\chi|\cdot|^{s}$ is the principal series whose elements are smooth functions $\Phi$ on ${\rm SL}_{2}(\mathbb{A}_{F})$ satisfying
$$
\Phi(n(b)m(a)g,s)=\chi(a)|a|^{s+1}\Phi(g,s)
$$
for $b\in \mathbb{A}_{F}$ and $a\in\mathbb{A}_{F}^{\times}$ with $n(b)m(a)\in B_{F}$, where
$$
n(b):=\begin{pmatrix}1&b\\0&1\end{pmatrix}\quad\mbox{and}\quad m(a):=\begin{pmatrix}a&0\\0&a^{-1}\end{pmatrix}
$$
and $B_{\mathbb{A}_{F}}$  the standard Borel subgroup of ${\rm SL}_{2}(\mathbb{A}_{F})$. Such an element is called factorizable if $\Phi=\otimes\Phi_{\nu}$ with $\Phi_{\nu}\in I(s,\chi)$. It is called standard if $\Phi_{{\rm SL}_{2}(\hat{\mathcal{O}}_{F}){\rm SO}_{2}(\mathbb{R})^{2}}$ is independent of $s$. For a standard element $\Phi$, its associated Eisenstein series is defined for $\Re(s)\gg0$ by
$$
E(g,s,\Phi)=\sum_{\gamma\in B_{\mathbb{A}_{F}}\backslash {\rm SL}_{2}(\mathbb{A}_{F})}\Phi(\gamma g,s).
$$

For $\phi\in\mathcal{S}(V_{\mathbb{A}_{\Q,f}})=\mathcal{S}(W_{\mathbb{A}_{F,f}})$, let $\Phi_{f}$ be the standard element associated to $\lambda_{f}(\phi)\in I(0,\chi)$. For each real embedding $\sigma_{i}$ of $F$, let $\Phi_{\sigma_{i}}\in I(s,\chi_{\mathbb{C}/\mathbb{R}})=I(s,\chi_{E_{\sigma_{i}}/F_{\sigma_{i}}})$ be the unique `weight one' eigenvector of ${\rm SL}_{2}(\mathbb{R})$ given by
$$
\Phi_{\sigma_{i}}(n(b)m(a)k_{\theta})=\chi_{\mathbb{C}/\mathbb{R}}(a)|a|^{s+1}e^{i\theta}
$$
for $b\in \mathbb{R}$, $a\in\mathbb{R}^{\times}$ and $k_{\theta}=\begin{pmatrix}\cos\theta&-\sin\theta\\\sin\theta&\cos\theta\end{pmatrix}\in {\rm SO}_{2}(\mathbb{R})$. We define for $\vec\tau=(\tau_{1},\ldots,\tau_{d+1})\in\mathbb{H}^{d+1}$,
$$
E(\vec\tau,s,\phi)=\left(\prod_{i=1}^{d+1}v_{i}\right)^{-\frac{1}{2}}E(g_{\vec\tau},s,\Phi_{f}\otimes\Phi_{\sigma_{1}}\otimes\Phi_{\sigma_{2}}),
$$
where $\tau_{i}=u_{i}+iv_{i}$ and $g_{\vec\tau}=(n(u_{i})m(\sqrt{v_{i}}))_{1\leq i\leq d+1}$. It is a non-holomorphic Hilbert modular form of scalar weight~1 for some congruence subgroup of ${\rm SL}_{2}(\mathcal{O}_{F})$. We normalize $E(\vec\tau,s,\phi)$ by
$$
E^{*}(\vec\tau,s,\phi)=\Lambda(s+1,\chi)E(\vec\tau,s,\phi)
$$
where 
$$
\Lambda(s,\chi)=\left({\rm N}_{F/\Q}(\partial_{F}d_{E/F})\right)^{\frac{s}{2}}\left(\pi^{-\frac{s+1}{2}}\Gamma\left(\frac{s+1}{2}\right)\right)^{d+1}L(s,\chi).
$$
We write the Fourier expansion of $E^{*}(\vec{\tau},s,\phi)$ as
\begin{equation}
\label{Et}
E^{*}(\vec{\tau},s,\phi)=E_{0}^{*}(\vec{\tau},s,\phi)+\sum_{t\in F^{\times}}E_{t}^{*}(\vec{\tau},s,\phi).
\end{equation}
If one assumes that $\phi$ is factorizable, then 
\begin{align}
\label{wtphi}
E^{*}_{t}(\vec{\tau},s,\phi)&=\prod_{\mathfrak{p}\nmid\infty}W^{*,\psi_{F}}_{t,\mathfrak{p}}(s,\phi)\prod_{i=1}^{d+1}W^{*}_{t,\sigma_{i}}(\tau_{i},s,\Phi_{\sigma_{i}})\\
\intertext{and}
\label{w00}
E_{0}^{*}(\vec{\tau},s,\phi)&=\phi(0)\Lambda(s+1,\chi)\left(\prod_{i=1}^{d+1}v_{i}\right)^{\frac{s}{2}}\\
&\quad+\left(\prod_{i=1}^{d+1}v_{i}\right)^{-\frac{s}{2}}\prod_{\mathfrak{p}<\infty}W_{0,\mathfrak{p}}^{*,\psi_{F}}(s,\phi)\prod_{i=1}^{d+1}\gamma(W_{\sigma_{i}}),
\end{align}
where
\begin{align}
\label{wtw}
W_{t,\mathfrak{p}}^{*,\psi_{F}}(s,\phi)&=|{\rm N}_{F/\Q}(\partial_{F}d_{E/F})|_{\mathfrak{p}}^{-\frac{s+1}{2}}L_{\mathfrak{p}}(1+s,\chi)W_{t,\mathfrak{p}}^{\psi_{F}}(s,\phi)
\intertext{and}
\label{wreal}
W^{*}_{t,\sigma_{i}}(\tau_{i},s,\Phi_{\sigma_{i}})&=v_{i}^{-1/2}\pi^{-\frac{s+2}{2}}\Gamma\left(\frac{s+2}{2}\right)W_{t,\sigma_{i}}^{\psi_{F}}(\tau_{i},s,\Phi_{\sigma_{i}}).
\end{align}
are the normalized local Whittaker functions, and $\gamma(W_{\sigma_{i}})$ are Weil indices (see, e.g., \cite{KRY, Ya}). Moreover, for $m>0$, we write
$$
a_{m}(\phi)=\sum_{\substack{t\in F_{+}^{\times}\\ {\rm tr}_{F/\mathbb{Q}}t=m}}a(t,\phi),
$$
where $F_{+}^{\times}$ consists of all totally positive elements in $F$, the term $a(t,\phi)$ is the $t$-th Fourier coefficient of $E^{*,'}(\tau^{\Delta},0,\phi)$ and $\tau^{\Delta}=(\tau,\ldots,\tau)$ the diagonal element, and write the constant term of $E^{*,'}(\vec\tau,s,\phi)$ as
$$
\phi(0)\Lambda(0,\chi)\log\left(\prod_{i=1}^{d+1}v_{i}\right)+a_{0}(\phi)
$$
for a constant $a_{0}(\phi)$ depending on $\phi$. One can check that $a(t,\phi)=0$ unless $t-Q_{F}(\mu)\in\partial_{F}^{-1}$ where $\mu$ is the element in $F$ associated with $\phi$ (see., e.g., \cite{KRY, Ya}).

Now we are ready to state the big CM value formula due to Bruinier, Kudla and Yang \cite[Thm.~5.2]{bky}, which expresses the sum of the values of a theta-lift on a Shimura variety $\mathcal{X}_{K}$ over a big CM cycle in terms of the coefficients of an incoherent Eisenstein series of weight~1.
\begin{Theorem}[Bruinier, Kudla, and Yang]
\label{bcm}
For a given $f\in M_{1-d,\omega_{L}}^{!}$ with 
$$
f(\tau)=\sum_{\mu\in L'/L}\sum_{\substack{n\in-Q(\mu)+\Z\\n\gg-\infty}}c(n,\mu)q^{n}\phi_{\mu},
$$
let $\Phi(Z,f)$ be the theta-lift defined as in Subsection~\ref{thetalift}. Then we have
$$
\sum_{Z\in Z(W)}\Phi(Z,f)=\frac{{\rm deg}(Z(W,z_{\sigma_{d+1}}^{\pm}))}{\Lambda(0,\chi)}\sum_{\substack{\mu\in L'/L\\m\in Q(\mu)+\mathbb{Z}\\m\geq0}}c(-m,\mu)a_{m}(\phi_{\mu}).
$$
\end{Theorem}

\subsection{Proof of Theorem \ref{gztype}}\label{pfgz} This subsection is devoted to the proof of our second main result, Theorem~\ref{gztype}. We first realize and interpret big CM cycles in the degenerate Hilbert modular surfaces for $\Gamma_{0}(N)$ under the identification $Y_{0}(N)\times Y_{0}(N)\cong \mathcal{X}_{K_{N}}$ as we have seen in Subsection~\ref{hil}. Then we conclude this section with the proof.

  Let $E_{i}=\mathbb{Q}(\sqrt{d_{i}})$ for $i=1,2$ with $(d_{1},d_{2})=1$ be two imaginary quadratic fields of fundamental discriminants $d_{i}$, and let $E=E_{1}\otimes E_{2}=\mathbb{Q}(\sqrt{d_{1}},\sqrt{d_{2}})$. Let $F=\mathbb{Q}(\sqrt{D})$ with $D=d_{1}d_{2}$ be the maximal totally real subfield of $E$. We can view $E$ as a $F$-quadratic space $W$ with $F$-quadratic form $Q_{F}(z)=\frac{z\bar{z}}{\sqrt{D}}$. Then we can also view it as a rational quadratic space ${\rm Res}_{F/\Q}W$ with quadratic form $Q(z)={\rm tr}_{F/\mathbb{Q}}\circ Q_{F}(z)$. Let $\sigma_{i}$ for $i=1,2$ be the two real embeddings of $F$ with $\sigma_{i}(\sqrt{D})=(-1)^{i-1}\sqrt{D}$. Then we can see that $W_{\sigma_{1}}$ has signature $(2,0)$ at $\sigma_{1}$ and $W_{\sigma_{2}}$ has signature $(0,2)$ at $\sigma_{2}$. We choose a $\mathbb{Z}$-basis for $\mathcal{O}_{E}$ as follows
$$
e_{1}=1\otimes1,\quad e_{2}=\frac{-d_{1}+\sqrt{d_{1}}}{2}\otimes1,\quad e_{3}=1\otimes\frac{d_{2}+\sqrt{d_{2}}}{2},\quad e_{4}=e_{2}e_{3},
$$
and throughout the remainder of this section, we will drop $\otimes$ when there is no ambiguity. Then we can identify $({\rm Res}_{F/\Q}W,Q(\cdot))$ with $(V=M_{2}(\mathbb{Q}),\det)$ considered in Subsection \ref{hil} by
$$
\sum_{i=1}^{4}x_{i}e_{i}\to\begin{pmatrix}x_{3}&x_{1}\\x_{4}&x_{2}\end{pmatrix}.
$$
In particular, under this identification, we have
\begin{equation}
\label{LNZ}
L_{N}\cong\mathbb{Z}+\mathbb{Z}\frac{-d_{1}+\sqrt{d_{1}}}{2}+\mathbb{Z}\frac{d_{2}+\sqrt{d_{2}}}{2}+\mathbb{Z}\frac{N(-d_{1}+\sqrt{d_{1}})(d_{2}+\sqrt{d_{2}})}{4}
\end{equation}
which is of index~$N$ in $\mathcal{O}_{E}$.

In such a case, the maximal torus $T$ over $\mathbb{Q}$ is given by (see \cite{hy} or \cite[Section 6]{bky})
$$
T(R)=\{(t_{1},t_{2})\in(E_{1}\otimes_{\mathbb{Q}}R)^{\times}\times(E_{2}\otimes_{\mathbb{Q}}R)^{\times}|\,t_{1}\bar{t}_{1}=t_{2}\bar{t}_{2}\}
$$
for any $\mathbb{Q}$-algebra $R$. Then the map from $T$ to $E$ is given by $(t_{1},t_{2})\to t_{1}/\bar{t}_{2}$. By the theory of complex multiplication \cite{shi}, there is an embedding
$$
\imath_{i}:E_{i}\to M_{2}(\mathbb{Q})
$$
such that
$$
\imath_{i}(t)\begin{pmatrix}e_{i+1}\\e_{1}\end{pmatrix}=\begin{pmatrix}te_{i+1}\\te_{1}\end{pmatrix}.
$$
Then $\imath=(\imath_{1},\imath_{2})$ gives the embedding from $T$ to $H$, and one has
$$
K_{N,T}:=K_{N}\cap T(\mathbb{Q})=\{(t_{1},t_{2})\in T(\mathbb{A}_{f})|\,t_{i}\in \imath^{-1}_{i}(K_{0}(N))\}.
$$
In the following, we will interpret the big CM cycle
$$
Z(W,z_{\sigma_{2}}^{\pm})= T(\mathbb{Q})\backslash \left(\{z_{\sigma_{2}}^{\pm}\}\times T(\mathbb{A}_{\Q,f})/K_{N,T}\right)
$$
in $\mathcal{X}_{K_{N}}$ as a 0-cycle in $Y_{0}(N)\times Y_{0}(N)$.

\begin{Lemma}
\label{ztau}
Under the identification between $\mathbb{D}$ and $\mathbb{H}^{2}\cup(\mathbb{H}^{-})^{2}$ (see Proposition \ref{yyprop}), the big CM points $z_{\sigma_{2}}$ and $z_{\sigma_{2}}^{-}$ are identified with $(\tau_{1},\tau_{2})\in\mathbb{H}^{2}$ and $(-\bar{\tau}_{1},-\bar{\tau}_{2})\in(\mathbb{H}^{-})^{2}$, respectively, where
$$
\tau_{i}=\frac{d_{i}+\sqrt{d_{i}}}{2}.
$$
\end{Lemma}
\begin{proof}
See \cite[Lemma 3.4]{yangyin}.
\end{proof}

\begin{Lemma}
\label{ringclass}
Let ${\rm Cl}_{N}(E_{i})$ be the ring class group of conductor $N$ of $E_{i}$ for $i=1,2$. Then there is an injection
$$
\jmath:T(\mathbb{Q})\backslash T(\mathbb{A}_{f})/K_{N,T}\hookrightarrow {\rm Cl}_{N}(E_{1})\times {\rm Cl}_{N}(E_{2})
$$
with image
\begin{align*}
S(N,d_{1},d_{2}):=&\left\{([\mathfrak{a}_{1}],[\mathfrak{a}_{2}])\in {\rm Cl}_{p}(E_{1})\times {\rm Cl}_{p}(E_{2})|\,\text{$[\mathfrak{a}_{i}]$ are representatives of ${\rm Cl}_{p}(E_{i})$ with ${\rm N}(\mathfrak{a}_{1})={\rm N}(\mathfrak{a}_{2})$}\right\}\\
\cong&\left\{\left.(Q_{1},Q_{2})\in\mathcal{Q}_{d_{1}}(N)/\Gamma_{0}(N)\times\mathcal{Q}_{d_{2}}(N)/\Gamma_{0}(N)\right|\,\mbox{$a_{1}=a_{2}$, i.e., $Q_{1}(1,0)=Q_{2}(1,0)$}\right\}
\end{align*}
where $\mathcal{Q}_{d}(N)$ denotes the set of primitive and positive definite binary quadratic forms $aX^{2}+bXY+cY^{2}$ of discriminant $d$ with $(a,N)=1$. The isomorphism is given by
$$
[aX^{2}+bXY+cY^{2}]\to \left[a,N\frac{-b+\sqrt{d}}{2}\right].
$$
\end{Lemma}

\begin{proof}
It is not hard to check that
$$
\imath_{i}\left(x+y\frac{d_{i}+\sqrt{d_{i}}}{2}\right)=\begin{pmatrix}x+yd_{i}&y\frac{d_{i}-d_{i}^{2}}{4}\\y&x\end{pmatrix}
$$
for $x,y\in\mathbb{Q}$. Thus, one has $\imath_{i}^{-1}(K_{0}(N))=\hat{\mathcal{O}}_{i,N}^{\times}$ where $\mathcal{O}_{i,N}$ is the order of conductor $N$ of $E_{i}$ and $E_{i}^{\times}\backslash\mathbb{A}_{E_{i},f}^{\times}/\imath_{i}^{-1}(K_{0}(N))\cong {\rm Cl}_{N}(E_{i})$ by \cite[Section 4.4]{bro}. This fact together with \cite[Lemma 3.5]{yangyin} implies the first assertion of the lemma. The isomorphism is due to Chen and Yui \cite[Thm.~4.4]{chenyui}.
\end{proof}

\begin{Proposition}
\label{galois}
Let $G_{i}$ be the associated Galois group of the ring class field of conductor $N$ of $E_{i}$. Then the points 
$$[z_{\sigma_{2}}^{\pm},(t_{1}^{-1},t_{2}^{-1})]\in  T(\mathbb{Q})\backslash \left(\{z_{\sigma_{2}}^{\pm}\}\times T(\mathbb{A}_{\Q,f})/K_{N,T}\right)$$ 
are identified with 
$$
[\tau_{1}^{\sigma_{\mathfrak{a}_{1}}},\tau_{2}^{\sigma_{\mathfrak{a}_{2}}}],\,[(-\bar{\tau}_{1})^{\sigma_{\mathfrak{a}_{1}}},(-\bar{\tau}_{2})^{\sigma_{\mathfrak{a}_{2}}}]\in Y_{0}(N)\times Y_{0}(N),
$$
respectively, where $\sigma_{\mathfrak{a}_{i}}\in G_{i}$ is the Galois element associated to $[\mathfrak{a}_{i}]\in {\rm Cl}_{N}(E_{i})$ via the Artin map, and $[\mathfrak{a}_{i}]\in {\rm Cl}_{N}(E_{i})$ is the ideal class associated to the idele class $[t_{i}]\in T(\mathbb{Q})\backslash T(\mathbb{A}_{\Q,f})/K_{N,T}$.

In particular, by the Shimura reciprocity law, one has
$
\tau_{i}^{\sigma_{\mathfrak{a}_{i}}}=\tau_{\mathfrak{a}_{i}}/N
$
 where $\tau_{\mathfrak{a}}$ is the CM point associated to the integral ideal $\mathfrak{a}$ of conductor $N$.
\end{Proposition}
\begin{proof}
The former results follow from Lemma \ref{ringclass} and \cite[Proposition 3.6]{yangyin}. The latter result follows from \cite[Chapter 6]{shi} (or see \cite{PS} for a nice summary of Shimura reciprocity law) and \cite[Theorem III]{chenyui}.
\end{proof}

Now we are ready for

\begin{proof}[{Proof of Theorem \ref{gztype}}]
By Theorem~\ref{bors}(2) and Theorem~\ref{borthm}, we deduce that
\begin{align}
\Phi(z,F_{N})&=-4\log|\Psi(z,F_{N})|\nonumber\\
\label{step1}
&=-4\log|\pi_{N}(z_{1})-\pi_{N}(z_{2})|
\end{align}
since $c(0,0)=0$ by Lemma~\ref{lemFN}. Then together with Proposition~\ref{galois} and Theorem~\ref{bcm},          
 equation~\eqref{step1} implies that
\begin{align}
&4\sum_{([\mathfrak{a}_{1}],[\mathfrak{a}_{2}])\in S(p,d_{1},d_{2})}\Bigg(\log|\pi_{N}(\tau_{1}^{\sigma_{\mathfrak{a}_{1}}})-\pi_{N}(\tau_{2}^{\sigma_{\mathfrak{a}_{2}}})|+\log|\pi_{N}((-\bar{\tau}_{1})^{\sigma_{\mathfrak{a}_{1}}})-\pi_{N}(\tau_{2}^{\sigma_{\mathfrak{a}_{2}}})|\nonumber\\
&\qquad\qquad\qquad+\log|\pi_{N}(\tau_{1}^{\sigma_{\mathfrak{a}_{1}}})-\pi_{N}((-\bar{\tau}_{2})^{\sigma_{\mathfrak{a}_{2}}})|+\log|\pi_{N}((-\bar{\tau}_{1})^{\sigma_{\mathfrak{a}_{1}}})-\pi_{N}((-\bar{\tau}_{2})^{\sigma_{\mathfrak{a}_{2}}})|\Bigg)\nonumber\\
&=-\frac{2|S(N,d_{1},d_{2})|}{\Lambda(0,\chi)}\left(a_{1}(\phi_{0,0})+\sum_{\substack{\mu\in L'/L}}c(0,\mu)a_{0}(\phi_{\mu})\right)\nonumber\\
\label{step2}
&=-\frac{|S(N,d_{1},d_{2})|w_{1}w_{2}}{2h(d_{1})h(d_{2})}\left(a_{1}(\phi_{0,0})+\sum_{\substack{\mu\in L'/L}}c(0,\mu)a_{0}(\phi_{\mu})\right),
\end{align}
where the simplification in the second equality follows from Lemma~\ref{lemFN} which tells that $c(-m,{\mu})=0$ except $c(-1,\mu_{0,0})=1$ for $m>0$, and the last equality follows from the fact that $\Lambda(0,\chi)=\Lambda(0,\chi_{E_{1}/\Q})\Lambda(0,\chi_{E_{2}/\Q})$. Since $\pi_{N}(\tau+1)=\pi_{N}(\tau)$ and $-\bar{\tau}_{i}=-d_{i}+\tau_{i}$, then the left hand side of~\eqref{step2} is simply
$$
16\sum_{([\mathfrak{a}_{1}],[\mathfrak{a}_{2}])\in S(p,d_{1},d_{2})}\log|\pi_{N}(\tau_{1}^{\sigma_{\mathfrak{a}_{1}}})-\pi_{N}(\tau_{2}^{\sigma_{\mathfrak{a}_{2}}})|,
$$
and thus we have
\begin{align}
&\sum_{([\mathfrak{a}_{1}],[\mathfrak{a}_{2}])\in S(p,d_{1},d_{2})}\log|\pi_{N}(\tau_{1}^{\sigma_{\mathfrak{a}_{1}}})-\pi_{N}(\tau_{2}^{\sigma_{\mathfrak{a}_{2}}})|\nonumber\\
&=-\frac{|S(N,d_{1},d_{2})|w_{1}w_{2}}{32h(d_{1})h(d_{2})}\left(a_{1}(\phi_{0,0})+\sum_{\substack{\mu\in L'/L}}c(0,\mu)a_{0}(\phi_{\mu})\right)\nonumber\\
&=-\frac{|S(N,d_{1},d_{2})|w_{1}w_{2}}{32h(d_{1})h(d_{2})}\left(\sum_{\substack{t\in F_{+}^{\times}\\ {\rm tr}_{F/\mathbb{Q}}t=m}}a(t,\phi_{0,0})+\sum_{\substack{\mu\in L'/L}}c(0,\mu)a_{0}(\phi_{\mu})\right)\nonumber\\
\label{step3}
&=-\frac{|S(N,d_{1},d_{2})|w_{1}w_{2}}{32h(d_{1})h(d_{2})}\left(\sum_{\substack{t=\frac{2m+D+\sqrt{D}}{2}\\ |2m+D|<\sqrt{D}\\m\in\Z}}a\left(\frac{t}{\sqrt{D}},\phi_{0,0}\right)+\sum_{\substack{\mu\in L'/L}}c(0,\mu)a_{0}(\phi_{\mu})\right),
\end{align}
where \eqref{step3} follows from the rescaling $t\to\frac{t}{\sqrt{D}}$ and the fact that $a\left(\frac{t}{\sqrt{D}},\phi_{0,0}\right)=0$ unless $\frac{t}{\sqrt{D}}~\in~\partial_{F}^{-1}~=~\frac{1}{\sqrt{D}}\mathcal{O}_{F}$.
Moreover, by Lemma~\ref{lemFN}, we can easily show that for $p\in\{3,5,7,13\}$, the constant terms $c(0,\mu_{0,k})=\frac{24}{p-1}$ for $1\leq k\leq p-1$ and $c(0,\mu)=0$ for $\mu\ne\mu_{0,k}$, and these simplify the right hand side of \eqref{step3} and yield \eqref{gzpi} in Theorem~\ref{gztype}. Finally, the quadratic form interpretation of the left hand side of~\eqref{gzpi} follows from the isomorphism given in Lemma~\ref{ringclass}.
\end{proof}


\section{Computations of $a_{1}\left(\frac{t}{\sqrt{D}},\phi_{0,0}\right)$ and $a_{0}(\phi_{0,k})$}
In this section, we assume that $p\in\{3,5,7,13\}$, and  we aim to compute $a_{1}\left(\frac{t}{\sqrt{D}},\phi_{0,0}\right)$ and $a_{0}(\phi_{0,k})$ explicitly. To compute $a_{1}\left(\frac{t}{\sqrt{D}},\phi_{0,0}\right)$, one needs to calculate the local Whittaker functions $W_{\frac{t}{\sqrt{D}},\mathfrak{p}}^{\psi_{F}}(s,\phi)$. We note that if one lets $\psi_{F}'(t)=\psi_{F}(\frac{t}{\sqrt{D}})$ and $W'=W$ with $F$-quadratic form $Q_{F}'(z)=z\bar{z}$, then the Weil representations associated to $(W,Q_{F},\psi_{F})$ and $(W',Q_{F}',\psi_{F}')$ are the same, and one has by \cite[Lemma 4.2.2]{hy}, 
\begin{equation}
\label{wwphi}
W_{\frac{t}{\sqrt{D}},\mathfrak{p}}^{\psi_{F}}(s,\phi)=|{D}|_{\mathfrak{p}}^{\frac{1}{2}}W_{t,\mathfrak{p}}^{\psi_{F}'}(s,\phi).
\end{equation}
Thus
\begin{equation}
\label{wdw}
\frac{W_{\frac{t}{\sqrt{D}},\mathfrak{p}}^{*}(s,\phi)}{\gamma(W_{\mathfrak{p}})}=|D|_{\mathfrak{p}}^{-\frac{s}{2}}L_{\mathfrak{p}}(s+1,\chi)\frac{W^{\psi_{F}'}_{t,\mathfrak{p}}(s,\phi)}{\gamma(W_{\mathfrak{p}}')},
\end{equation}
and we will compute $a_{1}\left(\frac{t}{\sqrt{D}},\phi_{0,0}\right)$ via $W^{\psi_{F}'}_{t,\mathfrak{p}}(s,\phi)$ whose calculations are tidier due to the normalization (see \cite[Sec. 6]{yyy}). In addition, for $a_{0}(\phi_{0,k})$, by the definition \eqref{w00}, one has 
$$
a_{0}(\phi_{0,k})=-\Lambda(0,\chi)\tilde{W}'_{0,f}(0,\phi_{0,k})=-h(E_{1})h(E_{2})\tilde{W}'_{0,f}(0,\phi_{0,k})
$$
with
$$
\tilde{W}_{0,f}(s,\phi_{\mu})=\prod_{\mathfrak{p}\nmid\infty}\frac{|D|_{\mathfrak{p}}^{-\frac{1}{2}}L_{\mathfrak{p}}(s+1,\chi)}{L_{\mathfrak{p}}(s,\chi)}\frac{W^{\psi_{F}}_{0,\mathfrak{p}}(s,\phi_{\mu})}{\gamma(W_{\mathfrak{p}})}
$$
provided $\phi_{\mu}$ is factorizable.

We will compute $a_{1}\left(\frac{t}{\sqrt{D}},\phi_{0,0}\right)$ and $a_{0}(\phi_{0,k})$ case by case according to the ramification of $p$ in $F$ and $E$. Throughout the remainder of this section, for $0\leq k\leq p-1$, we write $\phi_{0,k}$ for ${\rm Char}(\mu_{0,k}+L_{p}\otimes \hat{\mathbb{Z}})$. Clearly, by \eqref{LNZ}, one has
$$
 \phi_{0,k}=\phi_{0,k,p}\prod_{\mathfrak{p}\nmid p}\phi_{0,0,\mathfrak{p}}
$$
where $\phi_{0,k,p}={\rm Char}(\mu_{0,k}+\mathcal{L}_{p})$, $\mathcal{L}_{p}=L_{p}\otimes\mathbb{Z}_{p}$ and $\phi_{0,0,\mathfrak{p}}={\rm Char}(\mathcal{O}_{E_{\mathfrak{p}}})$ for $\mathfrak{p}\nmid{p}$. One key step to computing these coefficients via local Whittaker functions is the factorizability of $\phi_{\mu}$, so our main strategy is to write $\phi_{0,k,p}$ as a sum of products factorizable over $\mathfrak{p}|p$.
\begin{enumerate}
\item[Case 1.]{When $\left(\frac{d_{1}}{p}\right)=\left(\frac{d_{2}}{p}\right)=1$, then $p$ is completely split in $F$ and in $E$, that is, $p\mathcal{O}_{F}=\mathfrak{p}_{1}\mathfrak{p}_{2}$ and $p\mathcal{O}_{E}=\mathfrak{B}_{1}\bar{\mathfrak{B}}_{1}
\mathfrak{B}_{2}\bar{\mathfrak{B}}_{2}$. Similar to \cite[Section 5]{yangyin}, it is not hard to check that
$$
\mathcal{L}_{p}=\coprod_{i=0}^{p-1}(M_{i}\times M_{i})
$$
where $M_{i}=\left\{\left.(x_{1},x_{2})\in\Z_{p}^{2}\right|\,x_{1}+x_{2}\equiv i\pmod{p}\right\}$, and thus for $0\leq k\leq p-1$,
$$
 \phi_{0,k}=\sum_{i=0}^{p-1}\phi_{0,k}^{(i)}
$$
where 
$$
 {\phi_{0,k}^{(i)}=\phi^{(i)}_{\mathfrak{p}_{1}}\phi_{\mathfrak{p}_{2}}^{(i+k)}\prod_{\mathfrak{p}\nmid p}\phi_{0,0,\mathfrak{p}}}
$$
and $\phi^{(i)}={\rm Char}(M_{i})$. In this case, we aim to compute
$$
\sum_{\substack{t=\frac{2m+D+\sqrt{D}}{2}\\ |2m+D|<\sqrt{D}\\m\in\Z}}a\left(\frac{t}{\sqrt{D}},\phi_{0,0}\right)=a_{1}(\phi_{0,0})=\sum_{i=0}^{p-1}a_{1}(\phi_{0,0}^{(i)})
$$
where
$$
a_{1}(\phi_{0,0}^{(i)})=\sum_{\substack{t=\frac{2m+D+\sqrt{D}}{2}\\ |2m+D|<\sqrt{D}\\m\in\Z}}a\left(\frac{t}{\sqrt{D}},\phi_{0,0}^{(i)}\right),
$$
and for $1\leq k\leq p-1$,
$$
a_{0}(\phi_{0,k})=\sum_{i=0}^{p-1}a_{0}(\phi_{0,k}^{(i)})
$$
where
$$
a_{0}(\phi_{0,k}^{(i)})=-\tilde{W}_{0,f}'(0,\phi_{0,k}^{(i)}).
$$
Now we briefly explain how to compute $a\left(\frac{t}{\sqrt{D}},\phi_{0,0}^{(i)}\right)$. First denote by $\mbox{Diff}(W,t/\sqrt{D})$ the set of prime ideals $\mathfrak{p}$ of $F$ such that $W_{\mathfrak{p}}$ does not represent $t/\sqrt{D}$, i.e., $\mathfrak{p}\in\mbox{Diff}(W,t/\sqrt{D})$ if and only if $t\ne z\bar{z}$ for any $z\in E_{\mathfrak{p}}^{\times}$ if and only if $\mathfrak{p}$ is inert in $E/F$ and ${\rm ord}_{\mathfrak{p}}(t)$ is odd. By \cite[Prop. 2.7(1)]{yangyin} and \eqref{wtphi}, we can see that $a\left(\frac{t}{\sqrt{D}},\phi_{0,0}^{(i)}\right)=0$ when $|\mbox{Diff}(W,t/\sqrt{D})|>1$ since $W^{*}_{\frac{t}{\sqrt{D}},\mathfrak{p}}(0,\phi_{0,0}^{(i)})=0$ for $\mathfrak{p}\in\mbox{Diff}(W,t/\sqrt{D})$. Then when $\mbox{Diff}(W,t/\sqrt{D})=\{\mathfrak{p}\}$, by \cite[Prop. 2.7(2)]{yangyin} or direct calculations from \eqref{wtphi}, one has
\begin{equation}
\label{at0}
a\left(\frac{t}{\sqrt{D}},\phi_{0,0}^{(i)}\right)=-4W_{\frac{t}{\sqrt{D}},\mathfrak{p}}^{*,'}(0,\phi_{0,0}^{(i)})\prod_{\mathfrak{q}\nmid\mathfrak{p}\infty}W_{\frac{t}{\sqrt{D}},\mathfrak{q}}^{*}(0,\phi_{0,0}^{(i)}).
\end{equation}
Clearly, such a prime ideal $\mathfrak{p}$ must be neither $\mathfrak{p}_{1}$ nor $\mathfrak{p}_{2}$. In addition, by \cite[Prop. 2.7(3), (4)]{yangyin}, we have
$$
\frac{W_{\frac{t}{\sqrt{D}},\mathfrak{p}}^{*,'}(0,\phi_{0,0}^{(i)})}{\gamma(W_{\mathfrak{p}})}=\frac{1+{\rm ord}_{\mathfrak{p}}(t)}{2}\log N(\mathfrak{p})
$$
and
$$
\frac{W_{\frac{t}{\sqrt{D}},\mathfrak{q}}^{*}(0,\phi_{0,0}^{(i)})}{\gamma(W_{\mathfrak{q}})}=\rho_{\mathfrak{q}}(t\mathfrak{p}^{-1})
$$
for $\mathfrak{q}\ne\mathfrak{p}$, where $\rho_{\mathfrak{q}}(\mathfrak{a})$ is defined as in \eqref{localrho}.
 Then together with the facts that $\prod_{\mathfrak{q}\ne\mathfrak{p}}\rho_{\mathfrak{p}}(t\mathfrak{p}^{-1})=1$, we can rewrite \eqref{at0} as
\begin{align*}
a\left(\frac{t}{\sqrt{D}},\phi_{0,0}^{(i)}\right)&=-4\frac{1+{\rm ord}_{\mathfrak{p}}(t)}{2}\log N(\mathfrak{p})\prod_{\mathfrak{q}\nmid\mathfrak{p} p}\frac{W_{\frac{t}{\sqrt{D}},\mathfrak{q}}^{*}(0,\phi_{0,0}^{(i)})}{\gamma(W_{\mathfrak{q}})}\times\prod_{v}\gamma(W_{v})\times\prod_{j=1}^{2}\frac{W_{\frac{t}{\sqrt{D}},\mathfrak{p}_{j}}^{*}(0,\phi^{(i)}_{\mathfrak{p}_{j}})}{\gamma(W_{\mathfrak{p}_{j}})}\\
&=-4\frac{1+{\rm ord}_{\mathfrak{p}}(t)}{2}\log N(\mathfrak{p})\prod_{\mathfrak{q}\nmid p}\rho_{\mathfrak{q}}(t\mathfrak{p}^{-1})\prod_{j=1}^{2}\frac{W_{\frac{t}{\sqrt{D}},\mathfrak{p}_{j}}^{*}(0,\phi^{(i)}_{\mathfrak{p}_{j}})}{\gamma(W_{\mathfrak{p}_{j}})}\\
&=-4\frac{1+{\rm ord}_{\mathfrak{p}}(t)}{2}\log(N(\mathfrak{p}))\prod_{\mathfrak{q}\nmid p}\rho_{\mathfrak{q}}(t\mathfrak{p}^{-1})\prod_{j=1}^{2}L(1,\chi_{\mathfrak{p}_{j}})\frac{W_{t,\mathfrak{p}_{j}}^{\psi_{F}'}(0,\phi_{\mathfrak{p}_{j}}^{(i)})}{\gamma(W_{\mathfrak{p}_{j}}')}\\
&=-4{\frac{p^{2}}{(p-1)^{2}}}\frac{1+{\rm ord}_{\mathfrak{p}}(t)}{2}\log(N(\mathfrak{p}))\prod_{\mathfrak{q}\nmid p}\rho_{\mathfrak{q}}(t\mathfrak{p}^{-1})\prod_{j=1}^{2}\frac{W_{t,\mathfrak{p}_{j}}^{\psi_{F}'}(0,\phi_{\mathfrak{p}_{j}}^{(i)})}{\gamma(W_{\mathfrak{p}_{j}}')}
\end{align*}
where the third equality follows from \eqref{wdw}. Finally,  since this happens for exactly one prime ideal $\mathfrak{p}$ such that $\mbox{Diff}(W,t/\sqrt{D})=\{\mathfrak{p}\}$, we can write it in a unified form as
\begin{align*}
a\left(\frac{t}{\sqrt{D}},\phi_{0,0}^{(i)}\right)
&=-4{\frac{p^{2}}{(p-1)^{2}}}\sum_{\text{$\mathfrak{p}$ inert in $E/F$}}\frac{1+{\rm ord}_{\mathfrak{p}}(t)}{2}\log(N(\mathfrak{p}))\prod_{\mathfrak{q}\nmid p}\rho_{\mathfrak{q}}(t\mathfrak{p}^{-1})\prod_{j=1}^{2}\frac{W_{t,\mathfrak{p}_{j}}^{\psi_{F}'}(0,\phi_{\mathfrak{p}_{j}}^{(i)})}{\gamma(W_{\mathfrak{p}_{j}}')}.
\end{align*}
The local Whittaker functions $W_{t,\mathfrak{p}_{j}}^{\psi'_{F}}(0,\phi_{\mathfrak{p}_{j}}^{(i)})/\gamma(W'_{\mathfrak{p}_{j}})$ are computed explicitly in Corollary~\ref{evalw}.

For $a_{0}(\phi_{0,k}^{(i)})$, by definition, we first know that
\begin{align*}
&a_{0}(\phi_{0,k}^{(i)})\\
&=-\left.\left(\frac{L(1+s,\chi_{\mathfrak{p}_{1}})W_{0,\mathfrak{p}_{1}}^{\psi_{F}}(s,\phi_{\mathfrak{p}_{1}}^{(i)})}{L(s,\chi_{\mathfrak{p}_{1}})\gamma(W_{\mathfrak{p}_{1}})}\frac{L(1+s,\chi_{\mathfrak{p}_{2}})W_{0,\mathfrak{p}_{2}}^{\psi_{F}}(s,\phi_{\mathfrak{p}_{2}}^{(i+k)})}{L(s,\chi_{\mathfrak{p}_{2}})\gamma(W_{\mathfrak{p}_{2}})}\prod_{\mathfrak{p}\nmid p}\frac{L(1+s,\chi_{\mathfrak{p}})W_{0,\mathfrak{p}}^{\psi_{F}}(s,\phi_{0,0,\mathfrak{p}})}{L(s,\chi_{\mathfrak{p}})\gamma(W_{\mathfrak{p}})}\right)'\right|_{s=0}.
\end{align*}
By \cite[Prop. 5.7]{yyy}, one can check that
$$
\left.\left(\frac{L(1+s,\chi_{\mathfrak{p}})W_{0,\mathfrak{p}}^{\psi_{F}}(s,\phi_{0,0,\mathfrak{p}})}{L(s,\chi_{\mathfrak{p}})\gamma(W_{\mathfrak{p}})}\right)'\right|_{s=0}=0
$$
and
$$
\left.\frac{L(1+s,\chi_{\mathfrak{p}})W_{0,\mathfrak{p}}^{\psi_{F}}(s,\phi_{0,0,\mathfrak{p}})}{L(s,\chi_{\mathfrak{p}})\gamma(W_{\mathfrak{p}})}\right|_{s=0}=1.
$$
Thus, we can deduce that
$$
a_{0}(\phi_{0,k}^{(i)})=-\left.\left(\frac{L(1+s,\chi_{\mathfrak{p}_{1}})W_{0,\mathfrak{p}_{1}}^{\psi_{F}}(s,\phi_{\mathfrak{p}_{1}}^{(i)})}{L(s,\chi_{\mathfrak{p}_{1}})\gamma(W_{\mathfrak{p}_{1}})}\frac{L(1+s,\chi_{\mathfrak{p}_{2}})W_{0,\mathfrak{p}}^{\psi_{F}}(s,\phi_{\mathfrak{p}_{2}}^{(i+k)})}{L(s,\chi_{\mathfrak{p}})\gamma(W_{\mathfrak{p}})}\right)'\right|_{s=0},
$$
and by Lemma \ref{wst}, we know that for $i\ne0$
$$
\left.\left(\frac{L(1+s,\chi_{\mathfrak{p}_{j}})W_{0,\mathfrak{p}_{j}}^{\psi_{F}}(s,\phi_{\mathfrak{p}_{j}}^{(i)})}{L(s,\chi_{\mathfrak{p}_{j}})\gamma(W_{\mathfrak{p}_{j}})}\right)\right|_{s=0}=0
$$
and
$$
\left.\left(\frac{L(1+s,\chi_{\mathfrak{p}_{j}})W_{0,\mathfrak{p}_{j}}^{\psi_{F}}(s,\phi_{\mathfrak{p}_{j}}^{(i)})}{L(s,\chi_{\mathfrak{p}_{j}})\gamma(W_{\mathfrak{p}_{j}})}\right)'\right|_{s=0}=\frac{2\log{p}}{p-1}.
$$
Also, by Lemma \ref{wst}, one can check that
$$
\left.\left(\frac{L(1+s,\chi_{\mathfrak{p}_{j}})W_{0,\mathfrak{p}_{j}}^{\psi_{F}}(s,\phi_{\mathfrak{p}_{j}}^{(0)})}{L(s,\chi_{\mathfrak{p}_{j}})\gamma(W_{\mathfrak{p}_{j}})}\right)\right|_{s=0}=1.
$$
Therefore, we have
$$
a_{0}(\phi_{0,k})=-\frac{4\log{p}}{p-1}.
$$
}

\item[Case 2.]{When $\left(\frac{d_{1}}{p}\right)=\left(\frac{d_{2}}{p}\right)=-1$, then $p\mathcal{O}_{F}=\mathfrak{p}_{1}\mathfrak{p}_{2}$ and $\mathfrak{p}_{i}$ are inert in $E/F$. We have 
$$
E_{p}\cong E_{\mathfrak{p}_{1}}\times E_{\mathfrak{p}_{2}}
$$
 and the following identification (see, e.g., \cite[Prop. 4.31]{RV})
\begin{align*}
(\lambda_{\mathfrak{p}_{1}},\lambda_{\mathfrak{p}_{2}}):E_{p}&\to E_{\mathfrak{p}_{1}}\times E_{\mathfrak{p}_{2}}\\
\sqrt{d_{1}}&\to(\sqrt{d_{1}},\sqrt{d_{1}}),\\
\sqrt{d_{2}}&\to(\sqrt{d_{2}},-\sqrt{d_{2}}).
\end{align*}
Since $p\mathcal{O}_{E}\subset L_{p}$, then it is easy to see that
$$
\mathcal{L}_{p}=\coprod_{h,i,j\in\mathbb{Z}/p\mathbb{Z}}\left\{\left(\lambda_{\mathfrak{p}_{1}}(\gamma_{hij}),\lambda_{\mathfrak{p}_{2}}(\gamma_{hij})\right)+p\mathcal{O}_{E_{\mathfrak{p}_{1}}}\times p\mathcal{O}_{E_{\mathfrak{p}_{2}}}\right\}
$$
where
$$
\gamma_{hij}=h+i\frac{-d_{1}+\sqrt{d_{1}}}{2}+j\frac{d_{2}+\sqrt{d_{2}}}{2}.
$$
Thus, one has
$$
\phi_{0,k,p}=\sum_{h,i,j\in\mathbb{Z}/p\mathbb{Z}}\phi_{hij,\mathfrak{p}_{1}}^{(k)}\phi_{hij,\mathfrak{p}_{2}}^{(k)}
$$
where for $l=1,2$,
$$
\phi_{hij,\mathfrak{p}_{l}}^{(k)}={\rm Char}(\lambda_{\mathfrak{p}_{l}}(\gamma_{hij}^{(k)})+p\mathcal{O}_{E_{\mathfrak{p}_{l}}})
$$
and
$$
\gamma_{hij}^{(k)}=h+i\frac{-d_{1}+\sqrt{d_{1}}}{2}+j\frac{d_{2}+\sqrt{d_{2}}}{2}+k\frac{(-d_{1}+\sqrt{d_{1}})(d_{2}+\sqrt{d_{2}})}{4}.
$$
Then we can write
$$
\phi_{0,k}=\sum_{h,i,j\in\mathbb{Z}/p\mathbb{Z}}\Phi_{hij}^{(k)}
$$
where 
$$
\Phi_{hij}^{(k)}={\phi}_{hij,\mathfrak{p}_{1}}^{(k)}\phi_{hij,\mathfrak{p}_{2}}^{(k)}\prod_{\mathfrak{p}\nmid{p}}\phi_{0,0,\mathfrak{p}}.
$$
and we need to compute
$$
a_{1}(\phi_{0,0})=\sum_{h,i,j\in\mathbb{Z}/p\mathbb{Z}}a_{1}(\Phi_{hij}^{(0)})
$$
where
$$
a_{1}(\Phi_{hij}^{(0)})=\sum_{\substack{t=\frac{2m+D+\sqrt{D}}{2}\\ |2m+D|<\sqrt{D}\\m\in\Z}}a\left(\frac{t}{\sqrt{D}},\Phi_{hij}^{(0)}\right),
$$
and for $1\leq k\leq p-1$, 
$$
a_{0}(\phi_{0,k})=\sum_{h,i,j\in\mathbb{Z}/p\mathbb{Z}}a_{0}(\Phi_{hij}^{(k)})
$$
where
$$
a_{0}(\Phi_{hij}^{(k)})=-\tilde{W}'_{0,f}(0,\Phi^{(k)}_{hij}).
$$
Similar to Case~1, it is not hard to deduce that
\begin{align*}
&a\left(\frac{t}{\sqrt{D}},\Phi_{hij}^{(0)}\right)\\
&=-4\begin{cases}\displaystyle{\frac{p^{2}}{(p-1)^{2}}\frac{1+{\rm ord}_{\mathfrak{p}}(t)}{2}\log(N(\mathfrak{p}))\prod_{\mathfrak{q}\nmid{p}}\rho_\mathfrak{q}(t\mathfrak{p}^{-1})\prod_{l=1}^{2}\frac{W_{t,\mathfrak{p}_{l}}^{\psi_{F}'}(0,\phi_{hij,\mathfrak{p}_{l}}^{(0)})}{\gamma(W_{\mathfrak{p}_{l}}')}}&\text{if $\mathfrak{p}\nmid p$,}\\\\
\displaystyle{\frac{p}{p-1}\left.\left(\frac{p^{s+1}}{p^{s+1}-1}\frac{W_{t,{\mathfrak{p}_{t}}}^{\psi_{F}'}\left(s,{\phi}_{hij,\mathfrak{p}_{t}}^{(0)}\right)}{\gamma(W'_{{\mathfrak{p}_{t}}})}\right)'\right|_{s=0}\prod_{\mathfrak{q}\nmid p}\rho_{\mathfrak{q}}(t)\frac{W_{t,\tilde{\mathfrak{p}}_{t}}^{\psi_{F}'}(0,\phi_{hij,\tilde{\mathfrak{p}}_{t}}^{(0)})}{\gamma(W_{\tilde{\mathfrak{p}}_{t}}')}}&\text{if $\mathfrak{p}_{t}|p$ and above $t$,}
\end{cases}
\end{align*}
and for $1\leq k\leq p-1$,
\begin{align*}
&a_{0}(\Phi_{hij}^{(k)})\\
&=-\left.p^{2}\left(\prod_{l=1}^{2}\frac{(p^{s}-1)}{p^{s+1}-1}\frac{W_{0,\mathfrak{p}_{l}}^{\psi_{F}}(s,\phi_{hij,\mathfrak{p}_{l}}^{(k)})}{\gamma(W_{\mathfrak{p}_{l}})}\right)'\right|_{s=0}=-\left.p^{2}\left(\prod_{l=1}^{2}\frac{(p^{s}-1)}{p^{s+1}-1}\frac{W_{0,\mathfrak{p}_{l}}^{\psi'_{F}}(s,\phi_{hij,\mathfrak{p}_{l}}^{(k)})}{\gamma(W'_{\mathfrak{p}_{l}})}\right)'\right|_{s=0},
\end{align*}
where, by \cite[Cor. 5.3]{yyy},
\begin{align*}
&\frac{W_{t,{\mathfrak{p}_{l}}}^{\psi_{F}'}\left(s,{\phi}_{hij,\mathfrak{p}_{l}}^{(k)}\right)}{\gamma(W'_{{\mathfrak{p}_{l}}})}\\
&=\begin{cases}0&\mbox{if $o_{\mathfrak{p}_{l}}(p^{2}\eta_{hij,\mathfrak{p}_{l}}^{(k)}\bar{\eta}_{hij,\mathfrak{p}_{l}}^{(k)}-t)<0$,}\\\\
\displaystyle{(1-\frac{1}{p^{s}})\sum_{n=0}^{o_{\mathfrak{p}_{l}}(p^{2}\eta_{hij,\mathfrak{p}_{l}}^{(k)}\bar{\eta}_{hij,\mathfrak{p}_{l}}^{(k)}-t)}p^{n(1-s)}}&\mbox{if $0\leq o_{\mathfrak{p}_{l}}(p^{2}\eta_{hij,\mathfrak{p}_{l}}^{(k)}\bar{\eta}_{hij,\mathfrak{p}_{l}}^{(k)}-t)<o_{E_{\mathfrak{p}_{l}}}(\eta_{hij,\mathfrak{p}_{l}}^{(k)})$+2,}\\\\
\displaystyle{(1-\frac{1}{p^{s}})\sum_{n=0}^{o_{\mathfrak{p}_{l}}(p^{2}\eta_{hij,\mathfrak{p}_{l}}^{(k)}\bar{\eta}_{hij,\mathfrak{p}_{l}}^{(k)}-t)}p^{n(1-s)-2}}&\mbox{if $o_{\mathfrak{p}_{l}}(p^{2}\eta_{hij,\mathfrak{p}_{l}}^{(k)}\bar{\eta}_{hij,\mathfrak{p}_{l}}^{(k)}-t)\geq o_{E_{\mathfrak{p}_{l}}}(\eta_{hij,\mathfrak{p}_{l}}^{(k)})+2$,}\\
\qquad\qquad+p^{(o_{E_{\mathfrak{p}_{l}}}(\eta_{hij,\mathfrak{p}_{l}}^{(k)})+2)(1-s)-2}&\mbox{}\end{cases}
\end{align*}
and
$
\eta_{hij,\mathfrak{p}_{l}}^{(k)}=\frac{1}{p}\lambda_{\mathfrak{p}_{l}}(\gamma_{hij}^{(k)}),
$
$o_{\mathfrak{p}_{l}}(x)={\rm ord}_{\pi_{F_{\mathfrak{p}_{l}}}}(x)$, $\pi_{F_{\mathfrak{p}_{l}}}$ is the uniformizer of  $F_{\mathfrak{p}_{l}}$, $o_{E_{\mathfrak{p}_{l}}}(x)={\rm ord}_{\pi_{E_{\mathfrak{p}_{l}}}}(x)$, 
$\pi_{E_{\mathfrak{p}_{l}}}$ is the uniformizer of $E_{\mathfrak{p}_{l}}$. One can easily see that in this case,
$
a_{0}(\phi_{0,k})=0.
$
}


\item[Case 3.]{When $\left(\frac{d_{1}}{p}\right)=1$ and $\left(\frac{d_{2}}{p}\right)=-1$, then $p$ is inert in $F$ and is split in $E$, that is, $p\mathcal{O}_{F}=p$ and $p\mathcal{O}_{E}=\mathcal{B}\bar{\mathcal{B}}$. Then 

$$
E_{p}\cong F_{p}\times F_{p}
$$
and we have the following identification
\begin{align*}
\lambda:E_{p}&\to F_{p}\times F_{p}\\
\sqrt{d_{1}}&\to(\sqrt{d_{1}},-\sqrt{d_{1}}),\\
\sqrt{d_{2}}&\to(\sqrt{d_{2}},-\sqrt{d_{2}}).
\end{align*}
Clearly, the characteristic function $\phi_{k}={\rm Char}(\mu_{0,k}+L_{p}\otimes\hat{\mathbb{Z}})$ is factorizable over the spectrum of $F$, namely,
$$
\phi_{k}=\phi_{0,k,p}\prod_{\mathfrak{p}\nmid p}\phi_{0,0,\mathfrak{p}}
$$
where $\phi_{0,k,p}={\rm Char}(\mu_{0,k}+\mathcal{L}_{p})$ and $\phi_{0,0,\mathfrak{p}}={\rm Char}(\mathcal{O}_{E,\mathfrak{p}})$. Similar to Case~2, we have
$$
\phi_{0,k,p}=\sum_{h,i,j\in\mathbb{Z}/p\mathbb{Z}}\phi_{hij}^{(k)}
$$
where
$$
\phi_{hij}^{(k)}={\rm Char}(\lambda(\gamma_{hij}^{(k)})+p\mathcal{O}_{E_{p}})
$$
and
$$
\gamma_{hij}^{(k)}=h+i\frac{-d_{1}+\sqrt{d_{1}}}{2}+j\frac{d_{2}+\sqrt{d_{2}}}{2}+k\frac{(-d_{1}+\sqrt{d_{1}})(d_{2}+\sqrt{d_{2}})}{4}.
$$
And similarly, we write
$$
\phi_{0,k}=\sum_{h,i,j\in\mathbb{Z}/p\mathbb{Z}}\Phi_{hij}^{(k)}
$$
where 
$$
\Phi_{hij}^{(k)}={\phi}_{hij}^{(k)}\prod_{\mathfrak{p}\nmid p}\phi_{0,0,\mathfrak{p}}.
$$
and we aim to compute
$$
a_{1}(\phi_{0,0})=\sum_{h,i,j\in\mathbb{Z}/p\mathbb{Z}}a_{1}(\Phi_{hij}^{(0)})
$$
where
$$
a_{1}(\Phi_{hij}^{(0)})=\sum_{\substack{t=\frac{2m+D+\sqrt{D}}{2}\\ |2m+D|<\sqrt{D}\\m\in\Z}}a\left(\frac{t}{\sqrt{D}},\Phi_{hij}^{(0)}\right),
$$
and for $1\leq k\leq p-1$, 
$$
a_{0}(\phi_{0,k})=\sum_{h,i,j\in\mathbb{Z}/p\mathbb{Z}}a_{0}(\Phi_{hij}^{(k)})
$$
where
$$
a_{0}(\Phi_{hij}^{(k)})=-\tilde{W}'_{0,f}(0,\Phi^{(k)}_{hij}).
$$
Similar to Case~1, we can easily deduce that
$$
a\left(\frac{t}{\sqrt{D}},\Phi_{hij}^{(0)}\right)=-4\frac{p^{2}}{p^{2}-1}\sum_{\text{$\mathfrak{p}$ inert in $E/F$}}\frac{1+{\rm ord}_{\mathfrak{p}}(t)}{2}\log(N(\mathfrak{p}))\prod_{\mathfrak{q}\nmid p}\rho_{\mathfrak{q}}(t\mathfrak{p}^{-1})\frac{W_{t,p}^{\psi_{F}'}\left(0,{\phi}_{hij}^{(0)}\right)}{\gamma(W'_{p})}.
$$
Also, for $a_{0}(\Phi_{hij}^{(k)})$, we aim to compute
$$
a_{0}(\Phi_{hij}^{(k)})=-\left.\left(\frac{L(1+s,\chi_{p})}{L(s,\chi_{p})}\frac{W_{0,p}^{\psi_{F}}\left(s,{\phi}_{hij}^{(k)}\right)}{\gamma(W_{p})}\prod_{\mathfrak{p}\nmid p}\frac{L(1+s,\chi_{\mathfrak{p}})W_{0,\mathfrak{p}}^{\psi_{F}}(s,\phi_{0,0,\mathfrak{p}})}{L(s,\chi_{\mathfrak{p}})\gamma(W_{\mathfrak{p}})}\right)'\right|_{s=0}.
$$
Again, similar to Case~1, we can easily deduce that
$$
a_{0}(\Phi_{hij}^{(k)})=-\frac{2p^{2}\log{p}}{p^{2}-1}\frac{W_{0,p}^{\psi_{F}}\left(0,{\phi}_{hij}^{(k)}\right)}{\gamma(W_{p})}=-\frac{2p^{2}\log{p}}{p^{2}-1}\frac{W_{0,p}^{\psi_{F}'}\left(0,{\phi}_{hij}^{(k)}\right)}{\gamma(W'_{p})}
$$
where the second equality follows from \eqref{wwphi}. Finally, by \cite[Cor. 5.3]{yyy}, we have
$$
\frac{W_{t,p}^{\psi_{F}'}\left(0,{\phi}_{hij}^{(k)}\right)}{\gamma(W'_{p})}=\begin{cases}0&\mbox{if $o(p^{2}\eta_{hij}^{(k)}\bar{\eta}_{hij}^{(k)}-t)<o_{E_{p}}(\eta_{hij}^{(k)})+2$,}\\\\
p^{2o_{E_{p}}(\eta_{hij}^{(k)})+2}&\mbox{if $o(p^{2}\eta_{hij}^{(k)}\bar{\eta}_{hij}^{(k)}-t)\geq o_{E_{p}}(\eta_{hij}^{(k)})+2$,}\end{cases}
$$
where 
$
\eta_{hij}^{(k)}=\frac{1}{p}\gamma_{hij}^{(k)},
$
 $o(x)={\rm ord}_{\pi_{F_{p}}}(x)$, $\pi_{F_{p}}$ is the uniformizer of $F_{p}$, $o_{E_{p}}(x)=\min\{o(x_{1}),o(x_{2})\}$ for $x=(x_{1},x_{2})$ under the given identification $E_{p}=F_{p}\times F_{p}$.

}


\item[Case 4.]{When $p\nmid d_{1}$ and $\left(\frac{d_{2}}{p}\right)=1$, then $p\mathcal{O}_{F}=\tilde{\mathfrak{p}}^{2}$ and $\tilde{\mathfrak{p}}$ is split in $E/F$. Then we have $F_{p}\cong F_{\tilde{\mathfrak{p}}}$ and 
$$
E_{p}\cong F_{p}\times F_{p}
$$
with the following identification
\begin{align*}
\lambda:E_{p}&\to F_{p}\times F_{p}\\
\sqrt{d_{1}}&\to(\sqrt{d_{1}},-\sqrt{d_{1}}),\\
\sqrt{d_{2}}&\to(\sqrt{d_{2}},-\sqrt{d_{2}}).
\end{align*}
The treatment if essentially the same as Case~3, and in this case, we have
$$
a\left(\frac{t}{\sqrt{D}},\Phi_{hij}^{(0)}\right)=-4\frac{p}{p-1}\sum_{\text{$\mathfrak{p}$ inert in $E/F$}}\frac{1+{\rm ord}_{\mathfrak{p}}(t)}{2}\log(N(\mathfrak{p}))\prod_{\mathfrak{q}\nmid p}\rho_{\mathfrak{q}}(t\mathfrak{p}^{-1})\frac{W_{t,\tilde{\mathfrak{p}}}^{\psi_{F}'}\left(0,{\phi}_{hij}^{(0)}\right)}{\gamma(W'_{\tilde{\mathfrak{p}}})}
$$
and
$$
a_{0}(\Phi_{hij}^{(k)})=-\frac{p\log{p}}{p-1}\frac{W_{0,\tilde{\mathfrak{p}}}^{\psi_{F}'}\left(0,{\phi}_{hij}^{(k)}\right)}{\gamma(W'_{\tilde{\mathfrak{p}}})},
$$
where
$$
\frac{W_{t,\tilde{\mathfrak{p}}}^{\psi_{F}'}\left(0,\tilde{\phi}_{hij}^{(k)}\right)}{\gamma(W'_{\tilde{\mathfrak{p}}})}=\begin{cases}0&\mbox{if $o(p^{2}\eta_{hij}^{(k)}\bar{\eta}_{hij}^{(k)}-t)<o_{E_{p}}(\eta_{hij}^{(k)})$+4,}\\\\
p^{o_{E_{p}}(\eta_{hij}^{(k)})+2}&\mbox{if $o(p^{2}\eta_{hij}^{(k)}\bar{\eta}_{hij}^{(k)}-t)\geq o_{E_{p}}(\eta_{hij}^{(k)})+4$,}\end{cases}
$$
$$
\eta_{hij}^{(k)}=\frac{1}{p}\left(h+i\frac{-d_{1}+\sqrt{d_{1}}}{2}+j\frac{d_{2}+\sqrt{d_{2}}}{2}+k\frac{(-d_{1}+\sqrt{d_{1}})(d_{2}+\sqrt{d_{2}})}{4}\right),
$$
and $\pi_{F_{p}}$ is the uniformizer of $F_{p}$, $o_{E_{p}}(x)=\min\{o(x_{1}),o(x_{2})\}$ for $x=(x_{1},x_{2})$ under the given identification $E_{p}=F_{p}\times F_{p}$.

}


\item[Case 5.]{When $p\nmid d_{1}$ and $\left(\frac{d_{2}}{p}\right)=-1$, then $p\mathcal{O}_{F}=\tilde{\mathfrak{p}}^{2}$ and $\tilde{\mathfrak{p}}$ is inert in $E/F$. Similarly, the treatment of this case is essentially the same as the previous two, but one needs to consider an extra subcase, say $\mbox{Diff}(W,t/\sqrt{D})=\{\tilde{\mathfrak{p}}\}$. Thus we have
\begin{align*}
&a\left(\frac{t}{\sqrt{D}},\Phi_{hij}^{(0)}\right)\\
&=-4\begin{cases}\displaystyle{\frac{p}{p-1}\frac{1+{\rm ord}_{\mathfrak{p}}(t)}{2}\log(N(\mathfrak{p}))\prod_{\mathfrak{q}\nmid p}\rho_{\mathfrak{q}}(t\mathfrak{p}^{-1})\frac{W_{t,\tilde{\mathfrak{p}}}^{\psi_{F}'}\left(0,{\phi}_{hij}^{(0)}\right)}{\gamma(W'_{\tilde{\mathfrak{p}}})}}&\mbox{if $\mbox{Diff}(W,t/\sqrt{D})=\{\mathfrak{p}\}$ and $\mathfrak{p}\ne\tilde{\mathfrak{p}}$,}\\\\
\displaystyle{\left.\left(\frac{p^{2s+1}}{p^{s+1}-1}\frac{W_{t,\tilde{\mathfrak{p}}}^{\psi_{F}'}\left(s,{\phi}_{hij}^{(0)}\right)}{\gamma(W'_{\tilde{\mathfrak{p}}})}\right)'\right|_{s=0}\prod_{\mathfrak{q}\nmid p}\rho_{\mathfrak{q}}(t)}&\mbox{if $\mbox{Diff}(W,t/\sqrt{D})=\{\tilde{\mathfrak{p}}\}$,}\end{cases}
\end{align*}
and
$$
a_{0}(\Phi_{hij}^{(k)})=-\frac{p\log{p}}{p-1}\frac{W_{0,\tilde{\mathfrak{p}}}^{\psi_{F}'}\left(0,{\phi}_{hij}^{(k)}\right)}{\gamma(W'_{\tilde{\mathfrak{p}}})},
$$
where
\begin{align*}
&\frac{W_{t,\tilde{\mathfrak{p}}}^{\psi_{F}'}\left(s,{\phi}_{hij}^{(k)}\right)}{\gamma(W'_{\tilde{\mathfrak{p}}})}\\
&=\begin{cases}
0&\mbox{if $o(p^{2}\eta_{hij}^{(k)}\bar{\eta}_{hij}^{(k)}-t)<0$,}\\\\
\displaystyle{(1-\frac{1}{p^{s}})\sum_{n=0}^{o(p^{2}\eta_{hij}^{(k)}\bar{\eta}_{hij}^{(k)}-t)}p^{n(1-s)}}&\mbox{if $0\leq o(p^{2}\eta_{hij}^{(k)}\bar{\eta}_{hij}^{(k)}-t)<o_{E_{p}}(\eta_{hij}^{(k)})$+4,}\\\\
\displaystyle{(1-\frac{1}{p^{s}})\sum_{n=0}^{o(p^{2}\eta_{hij}^{(k)}\bar{\eta}_{hij}^{(k)}-t)}p^{n(1-s)-4}+p^{(o_{E_{p}}(\eta_{hij}^{(k)})+4)(1-s)-4}}&\mbox{if $o(p^{2}\eta_{hij}^{(k)}\bar{\eta}_{hij}^{(k)}-t)\geq o_{E_{p}}(\eta_{hij}^{(k)})+4$,}\end{cases}
\end{align*}
$$
\eta_{hij}^{(k)}=\frac{1}{p}\left(h+i\frac{-d_{1}+\sqrt{d_{1}}}{2}+j\frac{d_{2}+\sqrt{d_{2}}}{2}+k\frac{(-d_{1}+\sqrt{d_{1}})(d_{2}+\sqrt{d_{2}})}{4}\right),
$$
and $\pi_{F_{p}}$ is the uniformizer of $F_{p}$, $o_{E_{p}}(x)={\rm ord}_{\pi_{E_{\tilde{\mathfrak{p}}}}}(x)$, $\pi_{E_{\tilde{\mathfrak{p}}}}$ is the uniformizer of $E_{\tilde{\mathfrak{p}}}$.

}
\end{enumerate}

Finally, it remains to compute the local Whittaker functions $W_{t,\mathfrak{p}_{j}}^{\psi'_{F}}(0,\phi_{\mathfrak{p}_{j}}^{(i)})/\gamma(W'_{\mathfrak{p}_{j}})$ mentioned in Case~1. The following lemma essentially gives what we need. 
\begin{Lemma}
\label{wst}
Let $W=\mathbb{Q}_{p}^{2}$ with the quadratic form $Q(x)=x_{1}x_{2}$ and $\phi^{(i)}$ be defined as above. Let $N(i,m)=|\{j\in(\Z/p\Z)^{\times}|\,j(i-j)\equiv m\pmod{p}\}|$. Then the local Whittaker function $W_{t}(s,\phi^{(i)})$ for $t\in\Z_{p}$ is given by
\begin{align*}
\frac{W_{t}(s,\phi^{(0)})}{\gamma(W)}&=\begin{cases}\frac{1}{p}-\frac{1}{p^{1+s}}&\text{if $t\in k+p\Z_{p}$ and $\left(\frac{k}{p}\right)=1$,}\\\frac{1}{p}+\frac{1}{p^{1+s}}&\text{if $t\in k+p\Z_{p}$ and $\left(\frac{k}{p}\right)=-1$,}\\
\frac{1}{p}+\frac{p-1}{p}\sum_{n=2}^{{\rm ord}_{p}(t)}\frac{1}{p^{ns}}-\frac{1}{p^{1+(1+{\rm ord}_{p}(t))s}}&\text{if $t\in p\Z_{p}$,}\end{cases}\\
\intertext{and for $1\leq i\leq p-1$,}
\frac{W_{t}(s,\phi^{(i)})}{\gamma(W)}&=\begin{cases}\frac{1}{p}+\frac{1}{p^{1+s}}&\text{if $t\in p\Z_{p}$,}\\\frac{1}{p}+\frac{1}{p^{1+s}}(N(i,m)-1)&\text{if $t\in m+p\Z_{p}\subset\Z_{p}^{\times}$.}
\end{cases}
\end{align*}
\end{Lemma}

\begin{proof}
By the definition of $W_{t}(s,\phi)/\gamma(W)$ and unfolding, we have
\begin{align*}
\frac{W_{t}(s,\phi^{(i)})}{\gamma(W)}&=\int_{\mathbb{Q}_{p}}J_{i}(b)\psi(-tb)|a(wn(b))|^{s}db\\
&=\int_{\mathbb{Z}_{p}}J_{i}(b)\psi(-tb)db+\sum_{n=1}^{\infty}p^{n}\int_{\mathbb{Z}_{p}^{\times}}J_{i}(p^{-n}b)\psi(-p^{-n}tb)|a(wn(p^{-n}b))|^{s}db
\end{align*}
where
$$
J_{i}(b)=\int_{M_{i}}\psi(bx_{1}x_{2})dx_{1}dx_{2}.
$$
We first compute the integral $J_{i}(b)$. For $i=0$, we can deduce that
\begin{align*}
J_{0}(b)&=\int_{x_{1}+x_{2}\equiv0\pmod{p}}\psi(bx_{1}x_{2})dx_{1}dx_{2}\\
&=\frac{1}{p}\int_{\Z_{p}}\int_{\Z_{p}}\psi(b(py-x_{1})x_{1})dydx_{1}\\
&=\frac{1}{p}\int_{\Z_{p}}\psi(-bx_{1}^{2})\left(\int_{\Z_{p}}\psi(bpyx_{1})dy\right)dx_{1}\\
&=\frac{1}{p}\int_{\Z_{p}}\psi(-bx_{1}^{2}){\rm Char}(\Z_{p})(bpx_{1})dx_{1}.
\end{align*}
Now we compute the last integral case by case.
\begin{enumerate}
\item{For $b\in\Z_{p}$, 
\begin{align*}
J_{0}(b)&=\frac{1}{p}\int_{\Z_{p}}\psi(-bx_{1}^{2}){\rm Char}(\Z_{p})(bpx_{1})dx_{1}\\
&=\frac{1}{p}\int_{\Z_{p}}\psi(-bx_{1}^{2})dx_{1}\\
&=\frac{1}{p},
\end{align*}
}

\item{for $b\in\frac{k}{p}+\Z_{p}$ with $1\leq k\leq p-1$,
\begin{align*}
J_{0}(b)&=\frac{1}{p}\int_{\Z_{p}}\psi(-bx_{1}^{2}){\rm Char}(\Z_{p})(bpx_{1})dx_{1}\\
&=\frac{1}{p}\int_{\Z_{p}}\psi(-bx_{1}^{2})dx_{1}\\
&=\frac{1}{p}\left(\int_{p\Z_{p}}1dx_{1}+\int_{\Z_{p}^{\times}}\psi(-bx_{1}^{2})dx_{1}\right)\\
&=\frac{1}{p}\left(\frac{1}{p}+\sum_{j=1}^{p-1}\int_{j+p\Z_{p}}\psi(-bx_{1}^{2})dx_{1}\right)\\
&=\frac{1}{p}\left(\frac{1}{p}+\frac{1}{p}\sum_{j=1}^{p-1}e^{-\frac{2\pi ij^{2}k}{p}}\right),
\end{align*}
}

\item{for $b\not\in\frac{1}{p}\Z_{p}$,
\begin{align*}
J_{0}(b)&=\frac{1}{p}\int_{\Z_{p}}\psi(-bx_{1}^{2}){\rm Char}(\Z_{p})(bpx_{1})dx_{1}\\
&=\frac{1}{p}\int_{\frac{1}{bp}\Z_{p}}\psi(-bx_{1}^{2})dx_{1}\\
&=\frac{1}{p}\int_{\frac{1}{bp}\Z_{p}}1dx_{1}\\
&=|b|_{p}^{-1}.
\end{align*}
}
\end{enumerate}
In summary, one has
$$
J_{0}(b)=\begin{cases}\frac{1}{p}&\text{if $b\in\Z_{p}$,}\\\frac{1}{p^{2}}\left(1+\sum_{j=1}^{p-1}e^{-\frac{2\pi ij^{2}k}{p}}\right)&\text{if $b\in\frac{k}{p}+p\Z_{p}\subset\frac{1}{p}\Z_{p}^{\times}$,}\\
|b|_{p}^{-1}&\text{if $b\not\in\frac{1}{p}\Z_{p}$.}\end{cases}
$$
Now we are ready to compute
\begin{align*}
&\frac{W_{t}(s,\phi^{(0)})}{\gamma(W)}\\
&=\int_{\Z_{p}}J_{0}(b)\psi(-tb)db+\sum_{n=1}^{\infty}p^{n-ns}\int_{\Z_{p}^{\times}}J_{0}(p^{-n}b)\psi(-p^{-n}tb)db\\
&=\frac{1}{p}+p^{1-s}\sum_{j=1}^{p-1}\int_{j+p\Z_{p}}J_{0}(p^{-1}b)\psi(-p^{-1}tb)db+\sum_{n=2}^{\infty}p^{n-ns}\sum_{j=1}^{p-1}\int_{j+p\Z_{p}}J_{0}(p^{-n}b)\psi(-p^{-n}tb)db,
\end{align*}
where
\begin{enumerate}
\item{for $t\in -k+p\Z_{p}\subset\Z_{p}^{\times}$, 
\begin{align*}
&\frac{1}{p}+p^{1-s}\sum_{j=1}^{p-1}\int_{j+p\Z_{p}}J_{0}(p^{-1}b)\psi(-p^{-1}tb)db+\sum_{n=2}^{\infty}p^{n-ns}\sum_{j=1}^{p-1}\int_{j+p\Z_{p}}J_{0}(p^{-n}b)\psi(-p^{-n}tb)db\\
&=\frac{1}{p}+\frac{1}{p^{2+s}}\sum_{j=1}^{p-1}\left(1+\sum_{m=1}^{p-1}e^{-\frac{2\pi im^{2}j}{p}}\right)e^{\frac{2\pi ijk}{p}}\\
&=\frac{1}{p}+\frac{1}{p^{2+s}}\left(-1+\sum_{\substack{1\leq m\leq p-1\\m^{2}\equiv k\pmod{p}}}(p-1)-\sum_{\substack{1\leq m\leq p-1\\m^{2}\not\equiv k\pmod{p}}}1\right)\\
&=\begin{cases}\frac{1}{p}-\frac{1}{p^{1+s}}&\text{if $\left(\frac{k}{p}\right)=-1$,}\\\frac{1}{p}+\frac{1}{p^{1+s}}&\text{if $\left(\frac{k}{p}\right)=1$,}\end{cases}
\end{align*}}
\item{for $t\in p\Z_{p}$,
\begin{align*}
&\frac{1}{p}+p^{1-s}\sum_{j=1}^{p-1}\int_{j+p\Z_{p}}J_{0}(p^{-1}b)\psi(-p^{-1}tb)db+\sum_{n=2}^{\infty}p^{n-ns}\sum_{j=1}^{p-1}\int_{j+p\Z_{p}}J_{0}(p^{-n}b)\psi(-p^{-n}tb)db\\
&=\frac{1}{p}+\sum_{n=2}^{{\rm ord}_{p}(t)+1}\frac{1}{p^{ns}}\int_{\Z_{p}^{\times}}\psi(p^{-n+{\rm ord}_{p}(t)}b)db\\
&=\frac{1}{p}+\frac{p-1}{p}\sum_{n=2}^{{\rm ord}_{p}(t)}\frac{1}{p^{ns}}-\frac{1}{p^{1+(1+{\rm ord}_{p}(t))s}}.
\end{align*}}

\end{enumerate}
Therefore, one has
$$
\frac{W_{t}(s,\phi^{(0)})}{\gamma(W)}=\begin{cases}\frac{1}{p}-\frac{1}{p^{1+s}}&\text{if $t\in k+p\Z_{p}$ and $\left(\frac{k}{p}\right)=1$,}\\\frac{1}{p}+\frac{1}{p^{1+s}}&\text{if $t\in k+p\Z_{p}$ and $\left(\frac{k}{p}\right)=-1$,}\\
\frac{1}{p}+\frac{p-1}{p}\sum_{n=2}^{{\rm ord}_{p}(t)}\frac{1}{p^{ns}}-\frac{1}{p^{1+(1+{\rm ord}_{p}(t))s}}&\text{if $t\in p\Z_{p}$.}\end{cases}
$$
Now for $1\leq i\leq p-1$, similarly, we have
\begin{align*}
J_{i}(b)&=\int_{x_{1}+x_{2}\equiv i\pmod{p}}\psi(bx_{1}x_{2})dx_{1}dx_{2}\\
&=\frac{1}{p}\int_{\Z_{p}}\int_{\Z_{p}}\psi(bx_{1}(i+py-x_{1})dydx_{1}\\
&=\frac{1}{p}\int_{\Z_{p}}\psi(bx_{1}(i-x_{1})){\rm Char}(\Z_{p})(bpx_{1})dx_{1}.
\end{align*}
Then
\begin{enumerate}
\item{for $b\in\Z_{p}$,
\begin{align*}
J_{i}(b)&=\frac{1}{p}\int_{\Z_{p}}\psi(bx_{1}(i-x_{1})){\rm Char}(\Z_{p})(bpx_{1})dx_{1}\\
&=\frac{1}{p}\int_{\Z_{p}}1d{x_{1}}\\
&=\frac{1}{p},
\end{align*}
}

\item{for $b\in\frac{k}{p}+\Z_{p}\subset\frac{1}{p}\Z_{p}^{\times}$,
\begin{align*}
J_{i}(b)&=\frac{1}{p}\int_{\Z_{p}}\psi(bx_{1}(i-x_{1})){\rm Char}(\Z_{p})(bpx_{1})dx_{1}\\
&=\frac{1}{p}\int_{\Z_{p}}\psi(bx_{1}(i-x_{1}))dx_{1}\\
&=\frac{1}{p}\left(\int_{p\Z_{p}}1dx_{1}+\sum_{j=1}^{p-1}\int_{j+p\Z_{p}}\psi(bx_{1}(i-x_{1}))dx_{1}\right)\\
&=\frac{1}{p}\left(\frac{1}{p}+\frac{1}{p}\sum_{j=1}^{p-1}e^{\frac{2\pi i}{p}(kj(i-j))}\right),
\end{align*}
}

\item{for $b\in\frac{1}{p}\Z_{p}$,
\begin{align*}
J_{i}(b)&=\frac{1}{p}\int_{\Z_{p}}\psi(bx_{1}(i-x_{1})){\rm Char}(\Z_{p})(bpx_{1})dx_{1}\\
&=\frac{1}{p}\int_{\frac{1}{bp}\Z_{p}}\psi(bx_{1}(i-x_{1}))dx_{1}\\
&=\frac{1}{p}\left(\int_{\frac{1}{b}\Z_{p}}\psi(bx_{1}(i-x_{1}))dx_{1}+\int_{\frac{1}{bp}\Z_{p}^{\times}}\psi(bx_{1}(i-x_{1}))dx_{1}\right)\\
&=\frac{1}{p}\left(|b|_{p}^{-1}+\sum_{j=1}^{p-1}|b|_{p}^{-1}e^{\frac{2\pi i}{p}(ji)}\right)\\
&=0.
\end{align*}
}
\end{enumerate}
These can be summarized as follows.
$$
J_{i}(b)=\begin{cases}\frac{1}{p}&\text{if $b\in\Z_{p}$,}\\
\frac{1}{p^{2}}\left(1+\sum_{j=1}^{p-1}e^{\frac{2\pi i}{p}(kj(i-j))}\right)&\text{if $b\in\frac{k}{p}+\Z_{p}\subset\frac{1}{p}\Z_{p}^{\times}$,}\\
0&\text{if $b\not\in\frac{1}{p}\Z_{p}$.}
\end{cases}
$$
Then for $1\leq i\leq p-1$, $W_{t}(s,\phi^{(i)})/\gamma(W)$ can be computed in a similar way as $i=0$.
\begin{align*}
&\frac{W_{t}(s,\phi^{(i)})}{\gamma(W)}\\
&=\int_{\Z_{p}}J_{i}(b)\psi(-tb)db+\sum_{n=1}^{\infty}p^{n-ns}\int_{\Z_{p}^{\times}}J_{i}(p^{-n}b)\psi(-p^{-n}tb)db\\
&=\frac{1}{p}+\frac{1}{p^{1+s}}\sum_{k=1}^{p-1}\left(1+\sum_{j=1}^{p-1}e^{\frac{2\pi i}{p}(kj(i-j))}\right)\int_{-j+p\Z_{p}}\psi(p^{-1}tb)db,
\end{align*}
where
\begin{enumerate}
\item{for $t\in p\Z_{p}$,
\begin{align*}
&\frac{1}{p}+\frac{1}{p^{1+s}}\sum_{k=1}^{p-1}\left(1+\sum_{j=1}^{p-1}e^{\frac{2\pi i}{p}(kj(i-j))}\right)\int_{-j+p\Z_{p}}\psi(p^{-1}tb)db\\
&=\frac{1}{p}+\frac{1}{p^{2+s}}\sum_{k=1}^{p-1}\left(1+\sum_{j=1}^{p-1}e^{\frac{2\pi i}{p}(kj(i-j))}\right)\\
&=\frac{1}{p}+\frac{1}{p^{2+s}}\left(p-1-\sum_{\substack{1\leq j\leq p-1\\j\ne i}}1+p-1\right)\\
&=\frac{1}{p}+\frac{1}{p^{1+s}},
\end{align*}}

\item{for $t\in m+p\Z_{p}\subset \Z_{p}^{\times}$,
\begin{align*}
&\frac{1}{p}+\frac{1}{p^{1+s}}\sum_{k=1}^{p-1}\left(1+\sum_{j=1}^{p-1}e^{\frac{2\pi i}{p}(kj(i-j))}\right)\int_{-j+p\Z_{p}}\psi(p^{-1}tb)db\\
&=\frac{1}{p}+\frac{1}{p^{2+s}}\sum_{k=1}^{p-1}\left(1+\sum_{j=1}^{p-1}e^{\frac{2\pi i}{p}(kj(i-j))}\right)e^{-\frac{2\pi imk}{p}}\\
&=\frac{1}{p}+\frac{1}{p^{2+s}}\sum_{k=1}^{p-1}\left(e^{-\frac{2\pi imk}{p}}+\sum_{j=1}^{p-1}e^{\frac{2\pi i(j(i-j)-m)k}{p}}\right)\\
&=\frac{1}{p}+\frac{1}{p^{2+s}}\left(-1+\sum_{\substack{1\leq j\leq p-1\\j(i-j)\equiv m\pmod{p}}}(p-1)-\sum_{\substack{1\leq j\leq p-1\\j(i-j)\not\equiv m\pmod{p}}}1\right)\\
&=\frac{1}{p}+\frac{1}{p^{2+s}}\left(-1+(p-1)N(i,m)-(p-1)+N(i,m)\right)\\
&=\frac{1}{p}+\frac{1}{p^{1+s}}(N(i,m)-1).
\end{align*}}
\end{enumerate}
Therefore, we have
$$
\frac{W_{t}(s,\phi^{(i)})}{\gamma(W)}=\begin{cases}\frac{1}{p}+\frac{1}{p^{1+s}}&\text{if $t\in p\Z_{p}$,}\\\frac{1}{p}+\frac{1}{p^{1+s}}(N(i,m)-1)&\text{if $t\in m+p\Z_{p}\subset\Z_{p}^{\times}$.}
\end{cases}
$$
\end{proof}

Following from Lemma~\ref{wst}, the local Whittaker functions $W_{t,\mathfrak{p}_{j}}^{\psi'_{F}}(0,\phi_{\mathfrak{p}_{j}}^{(i)})/\gamma(W'_{\mathfrak{p}_{j}})$ can be computed explicitly as follows.

\begin{Corollary}
\label{evalw}
Fix $d\in(\Z/p\Z)^{\times}$ such that $\sqrt{D}\equiv d\pmod{p\mathbb{Z}_{p}}$, and fix the embeddings $
\sigma_{i}:F\hookrightarrow~F_{\mathfrak{p}_{i}}$ with $\mathfrak{p}_{1}\mathfrak{p}_{2}=p$. Write $t=\frac{2m+D+\sqrt{D}}{2}$ and $\tilde{d}=\frac{d^{2}+d}{2}$. Let $N(i,m)$ be defined as in Lemma~\ref{wst}. Then for $1\leq i\leq p-1$,
\begin{enumerate}
\item{when $m\equiv-\tilde{d}\pmod{p}$, 
$$
\frac{W_{t,\mathfrak{p}_{1}}^{\psi_{F}'}(0,\phi_{\mathfrak{p}_{1}}^{(i)})}{\gamma(W_{\mathfrak{p}_{1}}')}=\frac{2}{p}\quad\mbox{and}\quad \frac{W_{t,\mathfrak{p}_{2}}^{\psi_{F}'}(0,\phi_{\mathfrak{p}_{2}}^{(i)})}{\gamma(W_{\mathfrak{p}_{2}}')}=\frac{N(i,p-d)}{p};
$$
}

\item{when $m\equiv-\tilde{d}+d\pmod{p}$,
$$
\frac{W_{t,\mathfrak{p}_{1}}^{\psi_{F}'}(0,\phi_{\mathfrak{p}_{1}}^{(i)})}{\gamma(W_{\mathfrak{p}_{1}}')}=\frac{N(i,d)}{p}\quad\mbox{and}\quad \frac{W_{t,\mathfrak{p}_{2}}^{\psi_{F}'}(0,\phi_{\mathfrak{p}_{2}}^{(i)})}{\gamma(W_{\mathfrak{p}_{2}}')}=\frac{2}{p};
$$
}

\item{when $m\equiv-\tilde{d}+k\pmod{p}$ with $k\not\equiv0,d\pmod{p}$,
$$
\frac{W_{t,\mathfrak{p}_{1}}^{\psi_{F}'}(0,\phi_{\mathfrak{p}_{1}}^{(i)})}{\gamma(W_{\mathfrak{p}_{1}}')}=\frac{N(i,k)}{p}\quad\mbox{and}\quad \frac{W_{t,\mathfrak{p}_{2}}^{\psi_{F}'}(0,\phi_{\mathfrak{p}_{2}}^{(i)})}{\gamma(W_{\mathfrak{p}_{2}}')}=\frac{N(i,k-d)}{p},
$$
}
\end{enumerate}
and for $i=0$,
\begin{enumerate}
\item{when $m\equiv-\tilde{d}\pmod{p}$ and $\left(\frac{d}{p}\right)=-1$,
$$
\frac{W_{t,\mathfrak{p}_{1}}^{\psi_{F}'}(0,\phi_{\mathfrak{p}_{1}}^{(0)})}{\gamma(W_{\mathfrak{p}_{1}}')}=\frac{p-1}{p}({\rm ord}_{\mathfrak{p}_{1}}(\sigma_{1}(t))-1)\quad\mbox{and}\quad \frac{W_{t,\mathfrak{p}_{2}}^{\psi_{F}'}(0,\phi_{\mathfrak{p}_{2}}^{(0)})}{\gamma(W_{\mathfrak{p}_{2}}')}=0;
$$}

\item{when $m\equiv-\tilde{d}\pmod{p}$ and $\left(\frac{d}{p}\right)=1$,
$$
\frac{W_{t,\mathfrak{p}_{1}}^{\psi_{F}'}(0,\phi_{\mathfrak{p}_{1}}^{(0)})}{\gamma(W_{\mathfrak{p}_{1}}')}=\frac{p-1}{p}({\rm ord}_{\mathfrak{p}_{1}}(\sigma_{1}(t))-1)\quad\mbox{and}\quad \frac{W_{t,\mathfrak{p}_{2}}^{\psi_{F}'}(0,\phi_{\mathfrak{p}_{2}}^{(0)})}{\gamma(W_{\mathfrak{p}_{2}}')}=\frac{2}{p};
$$}

\item{when $m\equiv-\tilde{d}+d\pmod{p}$ and $\left(\frac{-d}{p}\right)=-1$,
$$
\frac{W_{t,\mathfrak{p}_{1}}^{\psi_{F}'}(0,\phi_{\mathfrak{p}_{1}}^{(0)})}{\gamma(W_{\mathfrak{p}_{1}}')}=0\quad\mbox{and}\quad \frac{W_{t,\mathfrak{p}_{2}}^{\psi_{F}'}(0,\phi_{\mathfrak{p}_{2}}^{(0)})}{\gamma(W_{\mathfrak{p}_{2}}')}=\frac{p-1}{p}({\rm ord}_{\mathfrak{p}_{2}}(\sigma_{2}(t))-1);
$$}

\item{when $m\equiv-\tilde{d}+d\pmod{p}$ and $\left(\frac{-d}{p}\right)=1$,
$$
\frac{W_{t,\mathfrak{p}_{1}}^{\psi_{F}'}(0,\phi_{\mathfrak{p}_{1}}^{(0)})}{\gamma(W_{\mathfrak{p}_{1}}')}=\frac{2}{p}\quad\mbox{and}\quad \frac{W_{t,\mathfrak{p}_{2}}^{\psi_{F}'}(0,\phi_{\mathfrak{p}_{2}}^{(0)})}{\gamma(W_{\mathfrak{p}_{2}}')}=\frac{p-1}{p}({\rm ord}_{\mathfrak{p}_{2}}(\sigma_{2}(t))-1);
$$}

\item{when $m\equiv-\tilde{d}+k$ with $k\not\equiv0,d\pmod{p}$, $\left(\frac{-k}{p}\right)=-1$ and $\left(\frac{d-k}{p}\right)=1$,
$$
\frac{W_{t,\mathfrak{p}_{1}}^{\psi_{F}'}(0,\phi_{\mathfrak{p}_{1}}^{(0)})}{\gamma(W_{\mathfrak{p}_{1}}')}=0\quad\mbox{and}\quad \frac{W_{t,\mathfrak{p}_{2}}^{\psi_{F}'}(0,\phi_{\mathfrak{p}_{2}}^{(0)})}{\gamma(W_{\mathfrak{p}_{2}}')}=\frac{2}{p};
$$}

\item{when $m\equiv-\tilde{d}+k$ with $k\not\equiv0,d\pmod{p}$, $\left(\frac{-k}{p}\right)=-1$ and $\left(\frac{d-k}{p}\right)=-1$,
$$
\frac{W_{t,\mathfrak{p}_{1}}^{\psi_{F}'}(0,\phi_{\mathfrak{p}_{1}}^{(0)})}{\gamma(W_{\mathfrak{p}_{1}}')}=0\quad\mbox{and}\quad \frac{W_{t,\mathfrak{p}_{2}}^{\psi_{F}'}(0,\phi_{\mathfrak{p}_{2}}^{(0)})}{\gamma(W_{\mathfrak{p}_{2}}')}=0;
$$}

\item{when $m\equiv-\tilde{d}+k$ with $k\not\equiv0,d\pmod{p}$, $\left(\frac{-k}{p}\right)=1$ and $\left(\frac{d-k}{p}\right)=1$,
$$
\frac{W_{t,\mathfrak{p}_{1}}^{\psi_{F}'}(0,\phi_{\mathfrak{p}_{1}}^{(0)})}{\gamma(W_{\mathfrak{p}_{1}}')}=\frac{2}{p}\quad\mbox{and}\quad \frac{W_{t,\mathfrak{p}_{2}}^{\psi_{F}'}(0,\phi_{\mathfrak{p}_{2}}^{(0)})}{\gamma(W_{\mathfrak{p}_{2}}')}=\frac{2}{p};
$$}

\item{when $m\equiv-\tilde{d}+k$ with $k\not\equiv0,d\pmod{p}$, $\left(\frac{-k}{p}\right)=1$ and $\left(\frac{d-k}{p}\right)=-1$,
$$
\frac{W_{t,\mathfrak{p}_{1}}^{\psi_{F}'}(0,\phi_{\mathfrak{p}_{1}}^{(0)})}{\gamma(W_{\mathfrak{p}_{1}}')}=\frac{2}{p}\quad\mbox{and}\quad \frac{W_{t,\mathfrak{p}_{2}}^{\psi_{F}'}(0,\phi_{\mathfrak{p}_{2}}^{(0)})}{\gamma(W_{\mathfrak{p}_{2}}')}=0.
$$}

\end{enumerate}

\end{Corollary}

\begin{proof}
These follow directly from Lemma \ref{wst}.
\end{proof}

{\bf Acknowledgment.} The author thanks Prof. Tonghai Yang for his guidance and encouragement, and he also thanks the anonymous referee for his/her useful comments, corrections and suggestions.

\end{document}